\documentclass{amsart}

\usepackage{amssymb}
\usepackage{amsthm}
\usepackage{amsmath}
\usepackage{graphicx}
\usepackage{txfonts}
\usepackage[all]{xy}
\usepackage{color}

\theoremstyle{plain}
\newtheorem{theorem}{Theorem}[section]
\newtheorem{proposition}[theorem]{Proposition}
\newtheorem{corollary}[theorem]{Corollary}
\newtheorem{lemma}[theorem]{Lemma}

\newtheorem*{conjecture*}{Conjecture}
\newtheorem*{thm:RuelleRtorsion}{Theorem~\ref{thm:Ruelle_Rtorsion_eqn}}
\theoremstyle{definition}
\newtheorem{definition}[theorem]{Definition}
\newtheorem*{definition*}{Definition}
\newtheorem*{example*}{Example}
\newtheorem*{notation*}{Notation}
\newtheorem*{notation-conv*}{Notation and convention}
\newtheorem*{convention*}{Convention}
\theoremstyle{remark}
\newtheorem{remark}[theorem]{Remark}
\newtheorem*{remark*}{Remark}

\newcommand{\N}{\mathbb{N}}
\newcommand{\Z}{\mathbb{Z}}

\newcommand{\R}{\mathbb{R}}
\newcommand{\C}{\mathbb{C}}

\newcommand{\upperH}{\mathbb{H}^2}
\newcommand{\re}{\mathop{\mathrm{Re}}\nolimits}
\newcommand{\img}{\mathop{\mathrm{Im}}\nolimits}
\newcommand{\iu}{{\scriptstyle \sqrt{-1}}}
\newcommand{\SL}[1][2]{\mathrm{SL}_{#1}(\C)}
\newcommand{\GL}{\mathrm{GL}_n(\C)}

\newcommand{\SO}{\mathrm{SO}(2)}

\newcommand{\sllR}{\mathfrak{sl}_2(\R)}

\newcommand{\SLR}[1][2]{{\mathrm{SL}_{#1}(\R)}}
\newcommand{\PSL}{{\mathrm{PSL}_2(\C)}}
\newcommand{\PSLR}{{\mathrm{PSL}_2(\R)}}
\newcommand{\TPSLR}{\widetilde{\mathrm{PSL}}_{2}(\R)}
\newcommand{\TSLR}{\widetilde{\mathrm{SL}}_{2}(\R)}

\newcommand{\trace}{\mathop{\mathrm{tr}}\nolimits}
\newcommand{\I}{\mathbf{1}}

\newcommand{\Isom}{\mathop{\mathrm{Isom}}}

\newcommand{\Tor}[2]{\mathop{\mathrm{Tor}}\nolimits (#1;#2)}

\newcommand{\ie}{i.e.,\,}
\begin{document}


\title[]{
  Dynamical zeta functions for geodesic flows
  and the higher-dimensional Reidemeister torsion
  for Fuchsian groups
}

\author{Yoshikazu Yamaguchi}

\address{Department of Mathematics,
  Akita University,
  1-1 Tegata-Gakuenmachi, Akita, 010-8502, Japan}
\email{shouji@math.akita-u.ac.jp}



\begin{abstract}
  We show that the absolute value at zero of
  the Ruelle zeta function defined by the geodesic flow 
  coincides with the higher-dimensional Reidemeister torsion
  for the unit tangent bundle over a $2$-dimensional hyperbolic orbifold
  and a non-unitary representation of the fundamental group.
  Our proof is based on the integral expression of the Ruelle zeta function.
  This integral expression is derived from the functional equation of the Selberg zeta function for
  a discrete subgroup with elliptic elements in $\PSLR$.
  We also show that the asymptotic behavior of the higher-dimensional Reidemeister torsion
  is determined by the contribution of the identity element
  to the integral expression of the Ruelle zeta function.
\end{abstract}


\maketitle

\section{Introduction}
We study the Ruelle zeta function defined for a dynamical system
called the geodesic flow on the unit tangent bundle over a $2$-dimensional hyperbolic orbifold
and the relation to the Reidemeister torsion for
the unit tangent bundle with non-unitary representations of the fundamental group.
Here a $2$-dimensional hyperbolic orbifold means the quotient of the hyperbolic plane
by a discrete subgroup of the orientation preserving isometry group, which is called a Fuchsian group. 
The Ruelle zeta function is closely related to
another zeta function called
the Selberg zeta function defined for a discrete subgroup of $\PSLR = \SLR / \{\pm \I\}$.
The values of the Ruelle zeta function have been studied 
by using the functional equation of the Selberg zeta function.
In particular, it has been revealed that
the (absolute) value at zero of the Ruelle zeta function 
determines the analytic torsion or the Reidemeister torsion of a manifold with an acyclic local system
given by a unitary representation of the fundamental group.
We refer the readers to the article~\cite{Shen18:FriedConj} by S.~Shen and
the references given there on those developments.

In this paper, we consider
the Ruelle zeta functions and the Reidemeister torsions defined for the unit tangent bundle $S(T\Sigma)$
over a $2$-dimensional hyperbolic orbifold $\Sigma$ 
and a sequence of $\SL[n]$-representations of the fundamental group $\pi_1(S(T\Sigma))$.
This study is motivated by the work~\cite{Muller:AsymptoticsAnalyticTorsion} of W.~M\"uller,
which investigates the asymptotic behavior of the analytic torsion or the Reidemeister torsion
for a closed oriented hyperbolic $3$-manifold $X$ and
$\SL[n]$-representations $\rho_n$ of $\pi_1(X)$ induced from the hyperbolic structure of $X$.
It was shown in~\cite{Muller:AsymptoticsAnalyticTorsion} that
the hyperbolic volume is derived from the limit of 
the leading coefficient in the logarithm of the analytic torsion or the Reidemeister torsion $\Tor{X}{\rho_n}$ 
for a closed hyperbolic $3$-manifold $X$ with $\SL[n]$-representations $\rho_n$ of $\pi_1(X)$.
More preciously, the leading coefficient converges as 
\[
\lim_{n \to \infty} \frac{\log |\Tor{X}{\rho_n}|}{n^2} = \frac{\mathrm{Vol}(X)}{4\pi}.
\]
Note that the sign of the right hand side is different from that of~\cite{Muller:AsymptoticsAnalyticTorsion}
since $\Tor{X}{\rho_n}$ in our convention is the inverse of that in~\cite{Muller:AsymptoticsAnalyticTorsion}.
These $\SL[n]$-representations $\rho_n$ are given by the symmetric power of the standard representation of $\SL$
and its restriction to the subgroup $\pi_1(X)$
when $\pi_1(X)$ is regarded as a discrete subgroup in $\PSL \simeq \Isom^+\mathbb{H}^3$
such that $X = \pi_1(X) \backslash \mathbb{H}^3$.
Here $\mathbb{H}^3$ is the hyperbolic $3$-space and $\Isom^+\mathbb{H}^3$ denotes the orientation preserving isometry group of $\mathbb{H}^3$.
The key ingredients of M\"uller's observation are the following:
\begin{itemize}
\item the absolute value at zero of the Ruelle zeta function equals to
  the square of the analytic torsion;
\item the Ruelle zeta function has a factorization whose factors have useful functional equations and;
\item the asymptotic behavior of the analytic torsion is analyzed by using
  the functional equations which the factors in the factorization of the Ruelle zeta function satisfy.
\end{itemize}
The functional equation of each factor in the Ruelle zeta function is derived from
that of the Selberg zeta function.
The functional equation of the Selberg zeta function includes the volume of the hyperbolic manifold.
Thus we find the hyperbolic volume in the asymptotic behavior of the analytic torsion.

The author has studied the asymptotic behavior of the Reidemeister torsion
for a Seifert fibered space and $\SL[2N]$-representations of the fundamental group
in~\cite{Yamaguchi:asymptoticsRtorsion}. 
The observations in~\cite{Yamaguchi:asymptoticsRtorsion} are based on explicit computations of
the Reidemeister torsion in terms of the decomposition of a Seifert fibered space along tori.
The purpose of this paper is to study the relation between the Ruelle zeta function of a dynamical system and 
a sequence of the Reidemeister torsion for a Seifert fibered space $M$ with $\SL[2N]$-representations of $\pi_1(M)$.
Throughout this paper, we consider a Seifert fibered space $M$ given by the quotient $\Gamma \backslash \PSLR$
by a discrete subgroup $\Gamma$ in $\PSLR$.
It is known that the quotient
$\Gamma \backslash \PSLR$ is regarded as the unit tangent bundle over the hyperbolic orbifold
$\Gamma \backslash \upperH$.
Here $\Gamma$ acts on $\upperH$
as a subgroup of the isometry group $\Isom^+ \upperH \simeq \PSLR$ by linear factorial transformations.

In the case of a hyperbolic $3$-manifold $X$,
there exists an injective homomorphism $\rho:\pi_1(X) \to \PSL \simeq \Isom^{+} \mathbb{H}^3$ with discrete image
such that $M = \rho(\pi_1(X)) \backslash \mathbb{H}^3$, which is called the holonomy representation.
By taking a lift of $\rho$ to a homomorphism $\pi_1(X) \to \SL$ and
the composition with the symmetric powers of the standard representation of $\SL$,
we have a sequence of the $\SL[n]$-representations $\rho_n$
of $\pi_1(X)$. Then we can define the Ruelle zeta function
and the analytic torsion or the Reidemeister torsion for $X$ with $\rho_n$.

We will deal with the unit tangent bundle $\Gamma \backslash \PSLR$
over a $2$-dimensional hyperbolic orbifold $\Gamma \backslash \upperH$
as a Seifert fibered space $M$.
Such a Seifert fibered space $M$ is also regarded as the quotient of the universal cover $\TPSLR$ of $\PSLR$
by $\pi_1(M)$ embedded in $\Isom \TPSLR$.
Actually the image of $\pi_1(M)$ by the embedding is contained in the subgroup $\TPSLR$ in $\Isom \TPSLR$.
We can regard the Seifert fibered space as the quotient
$\pi_1(M) \backslash \TPSLR$ by the subgroup $\pi_1(M) \subset \TPSLR$ and obtain the $\SLR$-representation
of $\pi_1(M)$ by the restriction of the projection $\TPSLR \to \SLR$ to $\pi_1(M)$.
Similarly to the case of a hyperbolic $3$-manifold, we can have a sequence of
$\SL[n]$-representations $\rho_n$ of $\pi_1(M)$
and define the Ruelle zeta functions $R_{\rho_{2N}}(s)$ and the Reidemeister torsions $\Tor{M}{\rho_{2N}}$
for $M$ with even dimensional representation $\rho_{2N}$.
We will show that 
\begin{theorem}[Theorem~\ref{thm:Ruelle_Rtorsion_eqn}]
  $|R_{\rho_{2N}}(0)| = \Tor{M}{\rho_{2N}}$.
\end{theorem}
Here the left hand side is the absolute value of the Ruelle zeta function at zero.
We show this equality as follows.
We see that
(the inverse of) the Ruelle zeta function at zero has the following integral expression
(Proposition~\ref{prop:eta_integral}):
\[
R_{\rho_{2N}}(0)^{-1}=
\exp\left(
\mu(\mathcal{D})\int_{0}^{N} \xi \tan \pi \xi\,d\xi
+ \sum_{\{R\}:\mathrm{elliptic}}
\int_{0}^{N} \frac{-\pi}{m(R)\sin\theta(R)}
\frac{\cos((2\theta(R)-\pi)\xi)}{\cos \pi\xi}\, d\xi \right)
\]
where $\mu(D)$ denotes the area of the base orbifold $\Gamma \backslash \upperH$,  
see Section~\ref{subsec:Ruelle_SL2N} for the details.
This integral expression of $R_{\rho_{2N}}(0)^{-1}$ is derived by
the functional equation of the Selberg zeta function in Theorem~\ref{thm:func_eq_Selberg_zeta}
for the discrete subgroup $\Gamma$ of $\PSLR$.
We allow $\Gamma$ to have torsion elements.
Then the absolute value $|R_{\rho_{2N}}(0)|^{-1}$ is expressed as
\begin{theorem}[Theorem~\ref{thm:Ruelle_zero}]
  \label{thm:intro_thm_II}
\begin{align*}
  |R_{\rho_{2N}}(0)|^{-1}
  &= \exp \left[ \mu(\mathcal{D})  \left( -\frac{N}{\pi} \log 2 \right) \right]\\
  &\quad \cdot
  \exp \left[
    \sum_{j=1}^m \left(
    \log \prod_{k=1}^N (1 - e^{(2k-1)\theta(q_j)\iu}) (1 - e^{-(2k-1)\theta(q_j)\iu})
    - \frac{2N}{m(q_j)}\log2
    \right)
    \right]
\end{align*}
where $\theta(q_j)$ and $m(q_j)$ are determined by
torsion elements $q_j$ in the fundamental group of the base orbifold.
\end{theorem}
\noindent
Comparing the result of $|R_{\rho_{2N}}(0)|$ with the Reidemeister torsion as in
Proposition~\ref{prop:Rtorsion_explicit},
we have the equality $|R_{\rho_{2N}}(0)| = \Tor{M}{\rho_{2N}}$.

Furthermore we observe that
the second factor in the right hand side of Theorem~\ref{thm:intro_thm_II}
does not affect in the limit of the leading coefficient
in $\log|R_{\rho_{2N}}(0)|$ as $N$ goes to infinity.
Hence only the first factor contributes the limit of the leading coefficient in $\log|R_{\rho_{2N}}(0)|$. 
Since it was shown in~\cite{Yamaguchi:asymptoticsRtorsion} that
the growth order of $\log|\Tor{M}{\rho_{2N}}| = \log|R_{\rho_{2N}}(0)|$ equals to $2N$,
this means that
\[
\frac{1}{2N}
\sum_{j=1}^m \left(
\log \prod_{k=1}^N (1 - e^{(2k-1)\theta(q_j)\iu}) (1 - e^{-(2k-1)\theta(q_j)\iu})
- \frac{2N}{m(q_j)}\log2
\right)
\xrightarrow{N \to \infty}
0
\]
and 
\[
  \lim_{N \to \infty} \frac{\log |\Tor{M}{\rho_{2N}}|}{2N}
  = \lim_{N \to \infty} \frac{\log |R_{\rho_{2N}}(0)|}{2N}
  = \frac{\mu(D)}{2\pi} \log 2. 
\]
We will discuss the above limits in Subsection~\ref{subsec:relation_asymptotic}.
By the Gauss-Bonnet theorem $\mu(D) = -2\pi \chi^{\mathrm{orb}}$, the following equality also holds:
\[
\lim_{N \to \infty} \frac{\log |\Tor{M}{\rho_{2N}}|}{2N}
= \lim_{N \to \infty} \frac{\log |R_{\rho_{2N}}(0)|}{2N}
=-\chi^{\mathrm{orb}}\log 2.
\]
Here $\chi^{\mathrm{orb}}$ is the orbifold Euler characteristic of the base orbifold $\Gamma \backslash \upperH$.
It was also shown in the previous work~\cite{Yamaguchi:asymptoticsRtorsion}
that $-\chi^{\mathrm{orb}}\log 2$ gives an upper bound of the limits
of the leading coefficients in the logarithm of the Reidemeister torsion for
a Seifert fibered space among the sequences of $\SL[2N]$-representations of its fundamental group.
For $M = \Gamma \backslash \PSLR$,
the sequence of $\SL[2N]$-representations induced from the geometric description of $M$
attains the upper bound in the limits
of the leading coefficients in the logarithm of the Reidemeister torsion for $M$
among sequences of $\SL[2N]$-representations of $\pi_1(M)$.

We note the differences from previous results on the Ruelle zeta function and the Reidemeister torsion.
D.~Fried observed 
the Ruelle zeta function and the Reidemeister torsion
for the unit sphere bundles $S(TX)$ of a closed oriented hyperbolic manifold $X$
with orthogonal representations of $\pi_1(S(TX))$ in~\cite{Fried86:AnalyticTorsionGeodesic}
and
especially for a Seifert fibered space $M = \Gamma \backslash \PSLR$ with
unitary representations of $\pi_1(M)$
in~\cite{Fried86:FuchsianGroupsRtorsion}.
We deal with the Ruelle zeta function and the Reidemeister torsion for non-unitary representations
of $\pi_1(M)$.
Fried derived the equality between the Ruelle zeta function and the Reidemeister torsion
by using the function equation of the Selberg zeta function.
We also derive the relation between the Ruelle zeta function and the Reidemeister torsion
by a functional equation of the Selberg zeta function.
Our functional equation is a different version from
those in~\cite{Fried86:AnalyticTorsionGeodesic, Fried86:FuchsianGroupsRtorsion}.

One can also find a number of developments on M\"uller's work
in the both sides of the analytic torsion and the Reidemeister torsion.
We will touch several developments
on the asymptotic behaviour of the analytic or Reidemeister torsion
for hyperbolic manifolds briefly.
P.~Menal-Ferrer and J.~Porti showed
the asymptotic behavior of the Reidemeister torsion
for non-compact hyperbolic $3$-manifolds of finite volume with cusps
in~\cite{FerrerPorti:HigherDimReidemeister}.
W\"uller and J.~Pfaff developed
the study of the asymptotic behavior of the analytic torsion for odd dimensional hyperbolic manifolds
of finite volume in~\cite{MuellerPfaff12:AsymptoticAnalyticNoncpt, MuellerPfaff13:AsymptoticRaySingerCpt}.
Pfaff also introduced in~\cite{Pfaff14:AnalyticvsRtorsion} the normalized analytic torsion 
corresponding to the higher-dimensional Reidemeister torsion invariant
introduced by Menal-Ferrer and Porti in~\cite{FerrerPorti:HigherDimReidemeister}
and compared them 
in detail.
H.~Goda and L.~B\'enard, J.~Dubois, M.~Heusener and J.~Porti provided
the asymptotic behavior of the higher-dimensional Reidemeister invariants
with the descriptions in terms of the twisted Alexander invariants.
These descriptions are given for hyperbolic knot complements in~\cite{Goda:TAP_hyp_vol} and
cusped hyperbolic $3$-manifolds in~\cite{BenardDuboisHeusenerPorti:AsymptoticTAP}. 

\subsection*{Organization}
Section~\ref{sec:preliminaries}
is devoted to reviews of the materials which we need in this paper.
Subsections~\ref{subsec:review_Seifert} and~\ref{subsec:geodesic_flow_Ruelle}
review the unit tangent bundle over a $2$-dimensional hyperbolic orbifold as a Seifert fibered space
and a dynamical system on the Seifert fibered space, called the geodesic flow.
The Subsections~\ref{subsec:Ruelle_Selberg_zeta},
\ref{subsec:review_trace_formula}, \ref{subsec:Selberg_transformation} and
\ref{subsec:review_functional_eqn} provide brief summaries
on the Ruelle zeta function, the Selberg zeta function, 
the Selberg trace formula and the functional equation
obtained as an application of the Selberg trace formula.
Section~\ref{sec:main_section} contains our main results.
Subsection~\ref{subsec:SLnRepSeifert} discusses the $\SL[2N]$-representations
given by the geometric description of our Seifert manifold.
We will see the explicit value of $|R_{\rho_{2N}}(0)|$ in Subsection~\ref{subsec:Ruelle_SL2N}
and the equality between $|R_{\rho_{2N}}(0)|$ and the Reidemeister torsion
in Subsection~\ref{subsec:Ruelle_Rtorsion_eq}.
Finally Subsection~\ref{subsec:relation_asymptotic}
presents the relation of $|R_{\rho_{2N}}(0)|$ to
the asymptotic behavior of  
the Reidemeister torsion for
a sequence of $2N$-dimensional representations as $N$ goes to infinity.

\section{Preliminaries}
\label{sec:preliminaries}

Let $\Gamma \subset \PSLR = \SLR / \{\pm \I\}$ be a discrete subgroup without parabolic elements.
We denote by $\upperH$ the upper half plane
which has the metric $(dx^2 + dy^2) / y^2$ with constant negative curvature $-1$.
We identify $\PSLR$ with the orientation preserving isometry group of $\upperH$
by linear fractional transformations.
The assumption on parabolic elements means that both of the manifold $M=\Gamma \backslash \PSLR$
and the orbifold $\Sigma = \Gamma \backslash \upperH$ are compact without boundary.

Throughout this paper,
a nontrivial element $g \in \PSLR$ is called {\it hyperbolic}, {\it elliptic} and {\it parabolic}
if a representative $A \in \SLR$ of $g$ has
the trace $|\trace A| > 2$, $|\trace A| < 2$ and $|\trace A|=2$ respectively.
Hence a representative of $g \in \PSLR$ is conjugate to one of the following matrices in $\SLR$:
\[
\pm \begin{pmatrix} e^{\ell/2} & 0 \\ 0 & e^{-\ell/2} \end{pmatrix},\quad 
\pm \begin{pmatrix} \cos\theta & -\sin\theta \\ \sin\theta & \cos\theta \end{pmatrix}, \quad
\begin{pmatrix} \pm 1 & * \\ 0 & \pm 1 \end{pmatrix}
\]
where $\ell$ is the translation length along the axis of a hyperbolic element $g$.
We often identify an element $g \in \PSLR$ with its representative in $\SLR$.

We will give brief reviews on $3$-dimensional manifolds $\Gamma \backslash \PSLR$
in Subsection~\ref{subsec:review_Seifert} and the Ruelle and Selberg zeta functions
associated with a dynamical flow on $\Gamma \backslash \PSLR$
in Subsections~\ref{subsec:geodesic_flow_Ruelle} and~\ref{subsec:Ruelle_Selberg_zeta}.
Subsections~\ref{subsec:review_trace_formula}, \ref{subsec:Selberg_transformation}
and~\ref{subsec:review_functional_eqn} give reviews on the Selberg trace formula and
the functional equation of the Selberg zeta function
in the case that a discrete subgroup $\Gamma$ has elliptic elements.

\subsection{Seifert Fibered space $\Gamma \backslash \PSLR$}
\label{subsec:review_Seifert}
It is known that the quotient of $\PSLR$ by a discrete subgroup $\Gamma$ is
a Seifert fibered space.
We review a $3$-dimensional manifold $\Gamma \backslash \PSLR$
as a Seifert fibered space.
We refer readers to \cite{scott83:3manifolds, NeumannJankins:Seifert} for the details.

Since every element $g$ of $\PSLR$ acts on $\upperH$ as an orientation preserving isometry,
each $g$ also acts on
the tangent bundle $T\upperH$ of $\upperH$ as an isometry
with metric induced from $\upperH$ .
Denote by $S(T\upperH)$ the {\it unit tangent bundle} of $\upperH$. 
Actually the action of $\PSLR$ on $S(T\upperH)$ is free,
\ie every stabiliser is trivial, and transitive
(for example,
the orbit of $z=\sqrt{-1}$ by $\PSLR$ is $\upperH$ and
the stabiliser of $z$ is given by $\SO$ 
which rotates unit tangent vectors in $S(T_z\upperH$)).
Thus we can regard $S(T\upperH)$ as the homogeneous space $\PSLR$.
The actions of a discrete subgroup $\Gamma$ in $\PSLR$ on $S(T\upperH) = \PSLR$ and $\upperH$
give the unit tangent bundle $\Gamma \backslash \PSLR$ over the orbifold $\Sigma=\Gamma \backslash \upperH$.
Since $S(T\Sigma) = \Gamma \backslash \PSLR$ is a circle bundle over $\Sigma$, we think of $\Gamma \backslash \PSLR$ as
a Seifert fibered space with the base orbifold $\Sigma$.

Under our assumption that $\Gamma$ has no parabolic elements, the singularities of the orbifold $\Sigma$ are only cone points.
If the orbifold $\Sigma$ has $m$ cone points with angle $2\pi / \alpha_j$
$(1 \leq j \leq m,\, \alpha_j \in \Z,\, \alpha_j \geq 2)$
and the underlying closed surface with genus $g$,
then $\Gamma = \pi_1(\Sigma)$ has the following
presentation:
\begin{equation}
  \label{eqn:presentation_Gamma}
  \Gamma
  = \langle a_1, b_1, \ldots, a_g, b_g, q_1, \ldots, q_m
  \,|\, q_j^{\alpha_j} = 1, \prod_{j=1}^m q_j \prod_{i=1}^2 [a_i, b_i] = 1\rangle
\end{equation}
where $a_i$ and $b_i$ are generators of the fundamental group of the underlying closed surface
and $q_j$ is the homotopy class of a closed loop around the $j$-th cone point.
The fundamental group of $\Gamma \backslash \PSLR$ has the following presentation:
\begin{align}
  &\pi_1(\Gamma \backslash \PSLR) \notag\\
  &= \langle a_1, b_1, \ldots, a_g, b_g, q_1, \ldots, q_m, h
  \,|\, h:\text{central}, q_j^{\alpha_j}h^{\alpha_j -1} = 1, \prod_{j=1}^m q_j \prod_{i=1}^g [a_i, b_i] = h^{2-2g}
  \rangle.
  \label{eqn:presentation_pi1_Seifert}
\end{align}
Here $h$ denotes the homotopy class of a regular fiber in the Seifert fibration of $\Gamma \backslash \PSLR$,
in other words, $h$ is the homotopy class of a fiber in $S(T\Sigma)$ on a smooth point of $\Sigma$.

It is also known that the exceptional fiber $\ell_j$ on the $j$-th cone point is expressed as
\begin{equation}
  \label{eqn:exceptional_fiber}
  \ell_j = q_j^{-1}h^{-1}.
\end{equation}

\begin{remark}
  The Seifert index of $\Gamma \backslash \PSLR$ is
  $(g; (1, 2g-2), (\alpha_1, \alpha_1-1), \ldots, (\alpha_m, \alpha_m-1))$.
  We follow the convention that the fundamental group $\pi_1(\Gamma \backslash \PSLR)$
  of a Seifert fibered space with index
  $(g; (1, b), (\alpha_1, \beta_1), \ldots, (\alpha_m, \beta_m))$ has the presentation:
  \[
  \langle a_1, b_1, \ldots, a_g, b_g, q_1, \ldots, q_m, h
  \,|\, h:\text{central}, q_j^{\alpha_j}h^{\beta_j} = 1, \prod_{j=1}^m q_j \prod_{i=1}^g [a_i, b_i] = h^{-b}
  \rangle
  \]
  and the exceptional fiber $\ell_j$ is expressed as
  $\ell_j = q_j^{\mu_j}h^{\nu_j}$ 
  for $\mu_j$ and $\nu_j \in \Z$ such that $\alpha_j \nu_j - \beta_j \mu_j = -1$.
\end{remark}

\subsection{Review on geodesic flows and the Ruelle zeta function}
\label{subsec:geodesic_flow_Ruelle}
We denote by $TX$ the tangent bundle over a closed Riemannian manifold $(X, g)$ and
by $S(TX)$ the unit tangent bundle.
Let $\phi_t$ be the following dynamical system on the unit tangent bundle $S(TX)$:
\begin{align*}
  \phi_t : S(TX) \times \R &\to S(TX) \\
  ((x, v), t) &\mapsto (\exp_x(tv), \frac{d}{dt} \exp_x(tv)),
\end{align*}
where $\exp:TX \to X$ is the exponential map given by the metric $g$.
This dynamical system is called the {\it geodesic flow} on $S(TX)$.

We can also consider the geodesic flow for a hyperbolic orbifold $\Sigma = \Gamma \backslash \upperH$
when $\Sigma$ has cone points.
There exists another definition of the geodesic flow on $S(T\Sigma)$
by a one parameter group action.
We adopt this alternative definition of the geodesic flow on $S(T\Sigma)$
and use it to define the related dynamical zeta function.

When we regard the manifold $\Gamma \backslash \PSLR$ as the unit tangent bundle $S(T\Sigma)$ of $\Sigma$,
the action of the one parameter group 
$a_t=\left\{\left. \begin{pmatrix} e^{t/2} & 0 \\ 0 & e^{-t/2} \end{pmatrix} \right| t \in \R \right\}$
on $\PSLR$ from the right, that is
\begin{align*}
  \PSLR \times \R &\to \PSLR \\
  (A, t) &\mapsto A\begin{pmatrix} e^{t/2} & 0 \\ 0 & e^{-t/2} \end{pmatrix},
\end{align*}
defines a dynamical system $\phi_t$ on $\PSLR$.
The one parameter group
$a_t$ defines the flow on $S(T\upperH) = \PSLR$ of the left invariant vector field generated by the element
$\begin{pmatrix} 1/2 & 0 \\ 0 & -1/2 \end{pmatrix}$ of the Lie algebra $\sllR$.
The orbit $g a_t$ for $g \in \PSLR$ projects to a geodesic through $g \cdot \sqrt{-1}$ on $\upperH$.
This action of the one parameter group $a_t$ also induces
the geodesic flow on $S(T\Sigma) = \Gamma \backslash \PSLR$.

Set $M = \Gamma \backslash \PSLR$. Then we can define the Ruelle zeta function
for the geodesic flow $\phi_t$ and a representation $\rho: \pi_1(M) \to \GL$
as follows.
\begin{definition}
  For $M=\Gamma \backslash \PSLR$ and a representation $\rho: \pi_1(M) \to \GL$, 
  the {\it Ruelle zeta function} $R_\rho(s)$ is defined as
  \[
  R_{\rho}(s) = \prod_{\gamma:\textrm{prime}} \det(\I - \rho(\gamma) e^{-s\ell(\gamma)})
  \]
  where $\gamma$ runs over the prime closed orbits of the geodesic flow $\phi_t$
  and $\ell(\gamma)$ denotes the period of $\gamma$.
\end{definition}
A closed orbit $\gamma$ of $\phi_t$ is called {\it prime} if $\gamma$ is not expressed as $\gamma = \gamma_0^{m}$
by any integer $m$ with $|m| > 1$ and other closed orbit $\gamma_0$.

The Ruelle zeta function is defined for $\re s \gg 0$ and
it can be extended to the complex plane meromorphically by
the meromorphic extension of the Selberg zeta function.
We will see the expression of the Ruelle zeta function in terms of the Selberg zeta functions
in Subsections~\ref{subsec:Ruelle_Selberg_zeta} and~\ref{subsec:Ruelle_SL2N}.

\begin{remark}
  One can define the Ruelle zeta function for more general dynamical systems
  which are called the {\it Anosov flows}.
  For the details, we refer to~\cite{Fried86:RuelleSelbergI, Shen18:FriedConj}
  on the Anosov flow and the Ruelle zeta function concerning the Reidemeister torsion.
  The alternative definition of the geodesic flow on $S(T\Sigma)$
  is referred to as an {\it algebraic Anosov flow}.
\end{remark}

Actually the period of $\gamma$ is thought of as the length of the closed geodesic on $\upperH$.
This is due to that 
the closed orbits of $\phi_t$ arise from the (unoriented) closed geodesics on $\Sigma = \Gamma \backslash \upperH$ by the definition of the geodesic flow.
Moreover we can regard a closed orbit $\gamma$ of $\phi_t$ as the conjugacy class of a hyperbolic element in $\Gamma$.
We review this correspondence shortly below.

Each closed orbit of the geodesic flow $\phi_t$ gives a closed geodesic on $\Sigma$.
If we lift geodesics on $\Sigma$ to $\upperH$,
the lifts are given by lines or arcs perpendicular to the real line.
Since such lines or arcs perpendicular to the real line are the axes of linear fractional transformations of hyperbolic elements in $\Gamma$, 
each closed geodesic on $\Sigma$ represents a hyperbolic element of $\Gamma$.
If we take a different starting point of a closed geodesic,
then the corresponding hyperbolic element of $\Gamma$ is changed by conjugation.
Hence there is a one--to--one correspondence between closed orbits of the geodesic flow $\phi_t$ and 
the conjugacy classes of hyperbolic elements in $\Gamma$.

Under correspondence between the closed orbits of the geodesic flow and
the conjugacy classes of hyperbolic elements in $\Gamma$, we can relate the Ruelle zeta function
to another zeta function called the Selberg zeta function.

\subsection{The Selberg zeta function}
\label{subsec:Ruelle_Selberg_zeta}
We review the Selberg zeta function for a discrete subgroup $\Gamma$ of $\PSLR$ and
the relation to the Ruelle zeta function for the geodesic flow on
the unit tangent bundle over $\Gamma \backslash \upperH$.
In what follows,  
$N(T)$ stands for the eigenvalue ratio $> 1$ of a representative matrix for a hyperbolic element $T \in \PSLR$,
which is referred to as the {\it norm} of $T$. 
\begin{definition}[the Selberg zeta function]
  The {\it Selberg zeta function} $Z(s)$ for a discrete subgroup $\Gamma$ of $\PSLR$ is defined by
  \begin{equation}
    \label{eqn:def_SelbergZeta}
  Z(s) = \prod_{\{T\}:\textrm{prime}} \prod_{k=0}^\infty (1-N(T)^{-s-k})
  \end{equation}
  for $\re(s) > 1$.
  Here $\{T\}$ runs over the conjugacy classes of nontrivial prime hyperbolic elements $T$ in $\Gamma$.
  A nontrivial element $T$ in $\Gamma$ is called {\it prime} if $T$ is not expressed as
  $T = T_0^m$ by any integer $m$ with $|m|>1$ and other element $T_0$.
\end{definition}
The right hand side of~\eqref{eqn:def_SelbergZeta} converges absolutely for  $\re(s) > 1$.
It is also known that $Z(s)$ can be extended to the complex plane meromorphically.
If $\gamma$ is the axis of the linear fractional transformation by
a hyperbolic element $T$ in $\Gamma$, then
$\gamma$ is a closed geodesic and the hyperbolic element $T$ is conjugate to
the diagonal matrix
$\begin{pmatrix} e^{\ell(\gamma)/2} & 0 \\ 0 & e^{-\ell(\gamma)/2} \end{pmatrix}$
under the correspondence between the conjugacy class $\{T\}$ and
the closed geodesic $\gamma$ in $\Gamma \backslash \upperH$.
Thus the norm $N(T)$ for the conjugacy class $\{T\}$ satisfies 
\[N(T) = e^{\ell(\gamma)}.\]

When we take the $1$-dimensional trivial representation $\rho : \pi_1(M) \to \mathrm{GL}_1(\C)$,
we denote the Ruelle zeta function by $R(s)$.
This classical Ruelle zeta function $R(s)$ is expressed as
\begin{align}
  R(s)
  &= \prod_{\gamma : \textrm{prime}} \det(1 - e^{-s\ell(\gamma)}) \notag\\
  &= \prod_{\{T\}:\textrm{prime}} \det(1 - N(T)^{-s}) \notag\\
  &= \frac{Z(s)}{Z(s+1)}. \label{eqn:classical_Ruelle_Selberg}
\end{align}

We will see a similar expression of the Ruelle zeta function for an $\SL[n]$-representation
of $\pi_1(\Gamma \backslash \upperH)$ in Section~\ref{sec:main_section}.

\subsection{Selberg Trace formula for a discrete subgroup with elliptic elements}
\label{subsec:review_trace_formula}
We will need the functional equation of the Selberg zeta function
for a discrete subgroup with elliptic elements in $\PSLR$.
The functional equation is derived from an application of the Selberg trace formula.
We review the Selberg trace formula for a discrete subgroup with elliptic elements
to make this article self-contained.
Since it is assumed that $\Gamma$ has no parabolic elements throughout the paper,
we do not consider orbifolds $\Gamma \backslash \upperH$ with cusps.

The Selberg trace formula gives a method of finding the trace of some linear operator
on the space of complex valued functions on $\Sigma = \Gamma \backslash \upperH$.
Here we identify functions on $\Sigma$ with functions $f(z)$ on $\upperH$ such that
$f(\sigma z) = f(z)$ for all $\sigma \in \Gamma$,
which are called {\it automorphic functions} on $\upperH$
with respect to $\Gamma$.

We consider linear operators $L$
commuting with the action of $\PSLR$ on functions $f(z)$ ($z \in \upperH$), that is,
$L(f(\sigma z)) = (Lf)(\sigma z)$. 
These linear operators are called {\it invariant operators}.
Each invariant operator acts on the space of automorphic functions with respect to $\Gamma$.
It is known that
the Laplacian $D = y^2 (\partial^2 / \partial x^2 + \partial^2 / \partial y^2)$ on $\upperH$
is also an invariant operator.
Under the assumption that $\Sigma$ has no cusps, 
the space $L^2(\Gamma \backslash \upperH)$ of square integrable automorphic functions
is a countably infinite dimensional vector space over $\C$ and has a basis consisting of
the eigenfunctions of the Laplacian $D$.
We restrict our attention to {\it integral operators} $L$, which are defined for a function $f(z)$
on $\upperH$ as
\[(Lf)(z) = \int_{\upperH}k(z, w)f(w)\,dw\]
with a kernel function $k(z, w) : \upperH \times \upperH \to \C$.
A kernel function $k(z, w)$ defines an invariant integral operator $L$
if and only if $k(z, w)$ satisfies $k(\sigma z, \sigma w) = k(z, w)$ for any $\sigma \in \PSLR$.
Such a kernel function is called a {\it point pair invariant}.

Set $K(z,w) = \sum_{\sigma \in \Gamma} k(z, \sigma w)$.
Then we regard the integral operator $L$
with a point pair invariant $k(z, w)$ on the space of automorphic functions 
as the integral operator with kernel $K(z, w)$ on the space of functions $f$ on $\Sigma$ by 
\[
L(f)(z)
= \int_{\upperH} k(z, w)f(w)\,d\mu(w)
= \int_{\mathcal{D}} \sum_{\sigma \in \Gamma} k(z, \sigma w) f(\sigma w)\,d\mu(w)
= \int_{\mathcal{D}} K(z, w)  f(w)\,d\mu(w)
\]
where $\mathcal{D}$ is a fundamental region of $\Sigma$ in $\upperH$.

The integral $\int_{\mathcal{D}} K(z,z)\,dz$ defined from a point pair invariant $k(z, w)$
is called the {\it trace} of the above integral operator $L$ defined by $K(z, w)$.
As in the finite dimensional case that
a linear map $f$ on a finite dimensional vector space is diagonalizable
if another linear map $g$ commuting with $f$ is diagonalizable,
it is known that $L$ on $L^2(\Gamma \backslash \upperH)$ commutes with the Laplacian $D$ and then
$L$ is diagonalized with respect to a basis
consisting of eigenfunctions $\{\varphi_n\}$ of $D$
such that
$D \varphi_n = \lambda_n \varphi_n$
($0 = \lambda_0 > \lambda_1 \geq \lambda_2 \geq \ldots \to -\infty$).
Each eigenfunction $\varphi_n$ of $D$ is also an eigenfunction of $L$ and
the eigenvalue is determined by $\lambda_n$
and the point pair invariant $k(z, w)$ defining $L$.
We write $\Lambda_k(\lambda_n)$ for the eigenvalue of an integral operator $L$
given by a point pair invariant $k(z, w)$
on $L^2(\Gamma \backslash \upperH)$
such that 
\[
L(\varphi_n) = \Lambda_k(\lambda_n)\varphi_n.
\]
The trace $\int_{D} K(z,z)\,d\mu(z)$ is thought of as
the sum of the eigenvalues $\Lambda_k(\lambda_n)$ $(n=0, 1, \ldots )$.
A.~Selberg has shown in~\cite{Selberg56:HarmonicAnalysis}
that this trace may be reduced to the sum of components corresponding to
conjugacy classes of $\Gamma$ and the eigenvalues $\Lambda_k(\lambda_n)$
and each component can be expressed by the function $h(r)$ made from $k(z, w)$.
The function $h(r)$ is called the {\it Selberg transform} of $k(z, w)$.

More precisely, the trace $\int_{D} K(z,z)\,d\mu(z)$ can be expressed as 
\[
\int_{\mathcal{D}} K(z,z)\,dz
= \sum_{\{\tau\}} \sum_{\sigma \in \{\tau\}} \int_{\mathcal{D}} k(z, \sigma z) \,d\mu(z)
\]
where $\{\tau\}$ means the conjugacy class of $\tau$ in $\Gamma$.
We can rewrite this equality as 
\begin{align}
\sum_{n = 0}^{\infty} \Lambda_k(\lambda_n)  
&= \int_{\mathcal{D}} k(z, \I z) \,d\mu(z) \label{eqn:equality_trace_L} \\
&\quad +\sum_{\{R\} : \mathrm{elliptic}} \sum_{\sigma \in \{R\}} \int_{\mathcal{D}} k(z, \sigma z) \,d\mu(z) \notag\\
&\quad +\sum_{\{T\} : \mathrm{hyperbolic}} \sum_{\sigma \in \{T\}} \int_{\mathcal{D}} k(z, \sigma z) \,d\mu(z) \notag
\end{align}
The first, second and third terms in the right hand side of~\eqref{eqn:equality_trace_L}
are the contributions of the identity, elliptic and hyperbolic elements
in $\Gamma$ respectively.
We will review that the both sides can be expressed by using the Selberg transformation $h(r)$
in the next subsection.

\subsection{The Selberg transformation}
\label{subsec:Selberg_transformation}
It is known that a kernel function $k(z, w)$ satisfies $k(\sigma z, \sigma w) = k(z, w)$
if and only if there exists a function $\Phi$ on $[0, \infty)$ such that
\[k(z, w) = \Phi\left( \frac{|z-w|^2}{\img(z)\img(w)} \right).\]

Let a point pair invariant $k(z, w)$ on $z$ and $w$ in $\upperH \times \upperH$ be
\[ k(z, w) = \Phi\left( \frac{|z-w|^2}{\img(z)\img(w)} \right) \]
for a $C^\infty$-function $\Phi$ with compact support on the interval $[0, \infty)$.
Set $Q(v)$ as
\[
Q(v) = \int_v^\infty \frac{\Phi(t)}{\sqrt{t-v}} \,dt
\]
and $g(u) := Q(e^u + e^{-u} -2)$.
Then we define the Selberg transform $h(r)$ of $k(z, w)$ as the Fourier transform of $g(u)$,
\ie
\[
h(r) := \int_{-\infty}^{\infty} g(u) e^{ru \iu}\,du.
\]

The Selberg transform $h(r)$ of $k(z, w)$ gives the eigenvalue $\Lambda_k(\lambda_n)$
of the integral operator $L$ defined by $k(z, w)$ as follows:
\[ \Lambda_k(\lambda_n) = h(r_n) \]
where $r_n$ satisfies $\lambda_n = -\frac{1}{4} - r_n^2$.

The value of $h(r_n)$ does not depend on the choice of $r_n$
since it is known that $h(r)$ is an even function. 
The sum of eigenvalues of the integral operator $L$ defined by a point pair invariant
$k(z, w)$ turns into the sum of
$h(r_n)$ in which we take $r_n$ as one of the complex numbers satisfying $\lambda_n = -\frac{1}{4} - r_n^2$
for the eigenvalue $\lambda_n$ of the Laplacian $D$ on $L^2(\Gamma \backslash \upperH)$.
We can rewrite each term in the equality~\eqref{eqn:equality_trace_L} with the Selberg transformation $h(r)$.
The resulting equality is called the {\it Selberg trace formula}
(we refer to~\cite{Selberg56:HarmonicAnalysis} and \cite[Appendix]{Kubota73:EisensteinSeries} for the details).

\begin{theorem}[The Selberg trace formula with elliptic elements]
  \label{thm:trace_formula}
  Suppose that $L$ is the integral operator on $L^2(\Gamma \backslash \upperH)$
  defined by the kernel $K(z, w) = \sum_{\sigma \in \Gamma} k(z, \sigma w)$
  given by a point pair invariant $k(z, w)$ on $\upperH \times \upperH$.
  Let $h(r)$ be the Selberg transform of $k(z, w)$.
  
  Then the equality~\eqref{eqn:equality_trace_L} on the trace of $L$ turns out to be
  \begin{align}
    \sum_{n=0}^\infty h(r_n)
    &= \frac{\mu(\mathcal{D})}{4\pi}\int_{-\infty}^{\infty} r h(r) \tanh(\pi r)\,dr
    \label{eqn:Selberg_trace_formula}\\
    &\quad + \sum_{\{R\}:\mathrm{elliptic}} \frac{1}{2m(R)\sin \theta (R)} \int_{-\infty}^\infty \frac{e^{-2 \theta (R) r}}{1+e^{-2\pi r}}h(r)\,dr \notag\\
    &\quad + \sum_{\{T\}:\mathrm{hyperbolic}} \frac{\log N(T_0)}{N(T)^{1/2} - N(T)^{-1/2}} g(\log N(T)) \notag
  \end{align}
  where
  \begin{itemize}
  \item
    we take $r_n$ as one of the complex numbers satisfying $\lambda_n = -\frac{1}{4} - r_n^2$
    for the eigenvalue $\lambda_n$ of the Laplacian $D$;
  \item
    $m(R)$ denotes the order of the centralizer $Z(R) = \{\sigma \in \Gamma \,|\, \sigma R \sigma^{-1} = R\}$;
  \item
    $\theta(R) \in (0, \pi)$ is determined by $\trace R = 2\cos \theta(R)$ and;
  \item
    $T_0$ is a prime hyperbolic element generating the centralizer
    \[Z(T) = \{ \sigma \in \Gamma \,|\, \sigma T \sigma^{-1} = T\}\]
    of $T$ in $\Gamma$.
  \end{itemize}
  
  The first, second and third terms in the right hand side are the contributions of
  the identity, elliptic and hyperbolic elements in $\Gamma$ respectively.
  The function $g(u)$ is also expressed as
  \[
  g(u) = \frac{1}{2\pi}\int_{-\infty}^\infty h(r)e^{-ur \iu }\,dr.
  \]  
\end{theorem}

Note that there exists a prime hyperbolic element $T_0$ for any hyperbolic element $T$ in $\Gamma$ 
such that $Z(T) = \langle T_0^n \,|\, n \in \Z \rangle$.
This means that
the centralizer $Z(T)$ of a hyperbolic element $T$ is conjugate to the cyclic group of infinite order
\[
\left\{\left.
\begin{pmatrix}
  e^{n \frac{\ell(\gamma_0)}{2}} & 0 \\
  0 & e^{-n \frac{\ell(\gamma_0)}{2}}
\end{pmatrix}
\in \PSLR
\, \right| \,
n \in \Z
\right\}
\]
where $\gamma_0$ is the closed geodesic corresponding to a prime element $T_0$.

\begin{remark}
The function $\Phi(t)$ defining $k(z, w)$ is recovered as
\[
\Phi(t) = \frac{-1}{\pi} \int_t^\infty \frac{dQ(v)}{\sqrt{v-1}}
\]
where
\[
Q(v)=g(u) \quad \text{for}\quad u = \log \frac{v+2 + \sqrt{v^2+4v}}{2}.
\]
\end{remark}

\subsection{The functional equation of the Selberg zeta function}
\label{subsec:review_functional_eqn}
We can also use the Selberg trace formula in Theorem~\ref{thm:trace_formula}
if an appropriate function $g(u)$ is given.
The purpose of this subsection is to
derive the functional equation of the Selberg trace formula
for a discrete subgroup with elliptic elements
along a way similar to that in~\cite[\S~1 -- 4 in Chap.~2]{Hejhal76:SelbergTraceFormulaVolI}.
We will focus on how elliptic elements affect to the functional equation of the Selberg
zeta function. For more details, we refer the reader to~\cite[Chap.~2]{Hejhal76:SelbergTraceFormulaVolI}.

Set a function $g(u)$ as
\[
g(u) = \frac{1}{2\alpha} e^{-\alpha|u|} - \frac{1}{2\beta}e^{-\beta|u|}
\]
for $u \in \R$ and $1/2 < \re(\alpha) < \re(\beta)$.
The Selberg transform $h(r)$ is expressed as
\[
h(r) = \frac{1}{r^2 + \alpha^2} - \frac{1}{r^2+\beta^2}.
\]
Then the Selberg trace formula~\eqref{eqn:Selberg_trace_formula} turns out to be  
\begin{align*}
  \sum_{n=0}^{\infty} \left( \frac{1}{r_n^2 + \alpha^2} - \frac{1}{r_n^2 + \beta^2} \right)
  &=
  \frac{\mu(\mathcal{D})}{4\pi}
  \int_{-\infty}^\infty
  r \left( \frac{1}{r^2 + \alpha^2} - \frac{1}{r^2 + \beta^2} \right) \tanh(\pi r)\,dr \\
  &\quad +
  \sum_{\{R\}:\mathrm{elliptic}} \frac{1}{2 m(R) \sin \theta(R)}
  \int_{-\infty}^\infty
  \frac{e^{-2\theta(R)r}}{1+e^{-2\pi r}}
  \left( \frac{1}{r^2 + \alpha^2} - \frac{1}{r^2 + \beta^2} \right)\,dr \\
  &\qquad +
  \frac{1}{2\alpha}
  \sum_{\{T\}: \mathrm{hyperbolic}}
  \frac{\log N(T_0)}{N(T)^{1/2} - N(T)^{-1/2}}\frac{1}{N(T)^\alpha} \\
  &\qquad -
  \frac{1}{2\beta}
  \sum_{\{T\}: \mathrm{hyperbolic}}
  \frac{\log N(T_0)}{N(T)^{1/2} - N(T)^{-1/2}}\frac{1}{N(T)^\beta} \\
\end{align*}

Here we replace the sums of terms for the conjugacy classes of hyperbolic elements $T$
with the Selberg zeta functions as follows (see~\cite[p.~66]{Hejhal76:SelbergTraceFormulaVolI} for the details):
\[
\sum_{\{T\}: \mathrm{hyperbolic}}
\frac{\log N(T_0)}{N(T)^{1/2} - N(T)^{-1/2}}\frac{1}{N(T)^\xi}
=
\frac{Z'(\xi+1/2)}{Z(\xi + 1/2)}
\]
Hence the Selberg trace formula~\eqref{eqn:Selberg_trace_formula} is expressed as  
\begin{align}
  &\sum_{n=0}^{\infty} \left( \frac{1}{r_n^2 + \alpha^2} - \frac{1}{r_n^2 + \beta^2} \right)
  \label{eqn:pre_functional_eqn}\\
  &=
  \frac{\mu(\mathcal{D})}{4\pi}
  \int_{-\infty}^\infty
  r \left( \frac{1}{r^2 + \alpha^2} - \frac{1}{r^2 + \beta^2} \right) \tanh(\pi r)\,dr \notag\\
  &\quad +
  \sum_{\{R\}:\mathrm{elliptic}} \frac{1}{2 m(R) \sin \theta(R)}
  \int_{-\infty}^\infty
  \frac{e^{-2\theta(R)r}}{1+e^{-2\pi r}}
  \left( \frac{1}{r^2 + \alpha^2} - \frac{1}{r^2 + \beta^2} \right)\,
  dr \notag\\
  &\quad
  +\frac{1}{2\alpha}
  \frac{Z'(\alpha+1/2)}{Z(\alpha + 1/2)}
  -\frac{1}{2\beta}
  \frac{Z'(\beta+1/2)}{Z(\beta + 1/2)} \notag.
\end{align}
Substituting $s-1/2$ into $\alpha$ in the above equality,
we can rewrite \eqref{eqn:pre_functional_eqn} as
\begin{align}
  &\frac{1}{2s-1}\frac{Z'(s)}{Z(s)} - \frac{1}{2\beta}\frac{Z'(\beta+1/2)}{Z(\beta + 1/2)}
  \label{eqn:pre_functional_eqn_II} \\
  &=
  \sum_{n=0}^{\infty} \left( \frac{1}{r^2 + (s-1/2)^2} - \frac{1}{r^2 + \beta^2} \right) \notag\\
  &\quad + \frac{\mu(\mathcal{D})}{4\pi}
  \int_{-\infty}^\infty
  r \left( \frac{1}{r^2 + \beta^2} - \frac{1}{r^2 + (s-1/2)^2} \right) \tanh(\pi r)\,dr \label{eqn:integral_identity}\\
  &\quad +
  \sum_{\{R\}:\mathrm{elliptic}} \frac{1}{2 m(R) \sin \theta(R)}
  \int_{-\infty}^\infty
  \frac{e^{-2\theta(R)r}}{1+e^{-2\pi r}}
  \left( \frac{1}{r^2 + \beta^2} - \frac{1}{r^2 + (s-1/2)^2} \right)\,dr \label{eqn:integral_elliptic}
\end{align}

The integrals of~\eqref{eqn:integral_identity} and~\eqref{eqn:integral_elliptic} can be carried out by the residue theorem,
for example, which follow from the integral 
along a rectangle with the vertices $X_1$, $X_1 + Y\sqrt{-1}$, $-X_2 + Y\sqrt{-1}$, $-X_2$
where $X_1$, $X_2 > 0$ and $Y \in \N$.
According to~\cite[pp.~69--71 and Proposition~4.9]{Hejhal76:SelbergTraceFormulaVolI},
the integral of~\eqref{eqn:integral_identity} is expressed as
\begin{align*}
\frac{\mu(\mathcal{D})}{4\pi}
  \int_{-\infty}^\infty
  r \left( \frac{1}{r^2 + \beta^2} - \frac{1}{r^2 + (s-1/2)^2} \right) \tanh(\pi r)\,dr
  &=
  \frac{\mu(\mathcal{D})}{2\pi}
  \sum_{k=0}^\infty \left(\frac{1}{\beta + 1/2 + k} - \frac{1}{s+k}\right).
\end{align*}
One can see that the integral of the third term~\eqref{eqn:integral_elliptic} turns into 
\begin{align}
  &\int_{-\infty}^\infty
  \frac{e^{-2\theta(R)r}}{1+e^{-2\pi r}}
  \left( \frac{1}{r^2 + \beta^2} - \frac{1}{r^2 + (s-1/2)^2} \right)\,dr \notag\\
  &=
  2\pi \sqrt{-1}
  \left(
  \sum_{k=0}^{\infty}
  \frac{1}{2\pi} e^{- \theta(R) (2k+1) \iu }
  \left(
  \frac{1}
       {
         \left( \frac{2k+1}{2} \iu \right)^2 + \beta^2
       }
  - \frac{1}
       {
         \left( \frac{2k+1}{2} \iu \right)^2 + (s-\frac{1}{2})^2
       }
  \right)
  \right)
  \label{eqn:residues_not_h}\\
  & \quad + 2\pi \sqrt{-1}
  \left(
   \frac{1}{2\sqrt{-1} \beta}\frac{e^{-2\theta(R)\iu \beta}}{1+e^{-2\pi \iu  \beta}}
  - \frac{1}{2\sqrt{-1} (s-1/2)}\frac{e^{-2\theta(R)(s-1/2) \iu}}{1+e^{-2\pi(s-1/2)\iu}}
  \right)
  \label{eqn:residues_h}
\end{align}
by calculating the residues of the integrand. 
The sums~\eqref{eqn:residues_not_h} and \eqref{eqn:residues_h} correspond to
the residues of $e^{-2\theta(R)z}/ (1+e^{-2\pi z})$ and $1/(z^2 + \beta^2) - 1/(z^2 + (s-1/2)^2)$
respectively.

If $\beta$ moves to $1/2 - s$, then the whole equality~\eqref{eqn:pre_functional_eqn_II} turns out to be 
\begin{align*}
  &\frac{1}{2s-1}\left(\frac{Z'(s)}{Z(s)} + \frac{Z'(1-s)}{Z(1-s)} \right)\\
  &=
  \frac{\mu(\mathcal{D})}{2\pi}
  \sum_{k=0}^\infty \left(\frac{1}{1 - s + k} - \frac{1}{s+k}\right) \\
  &\quad +
  \sum_{\{R\}:\mathrm{elliptic}} \frac{\pi}{2 m(R) \sin \theta(R)} 
  \left(
   \frac{1}{1/2 -s}\frac{e^{-2\theta(R)\iu (1/2 -s)}}{1+e^{-2\pi \iu (1/2 - s)}}
  - \frac{1}{s-1/2}\frac{e^{-2\theta(R)(s-1/2) \iu}}{1+e^{-2\pi(s-1/2)\iu}}
  \right).\\
\end{align*}
The first sum in the right hand side can be rewritten as
\[
  \sum_{k=0}^\infty \left(\frac{1}{1 - s + k} - \frac{1}{s+k}\right)
  = -\pi \cot (\pi s)\\
  = \pi \tan \pi(s-1/2).
\]
Therefore we have
\begin{align}
  \frac{1}{2s-1}\left(\frac{Z'(s)}{Z(s)} + \frac{Z'(1-s)}{Z(1-s)} \right)
  &=\frac{\mu(\mathcal{D})}{2} \tan\pi(s-1/2) \label{eqn:pre_functional_eqn_III}\\
  &\quad + \sum_{\{R\}:\mathrm{elliptic}}
  \frac{-\pi}{(2s-1)m(R) \sin \theta(R)} \frac{\cos((2\theta(R)-\pi)(s-1/2))}{\cos \pi(s-1/2)}.
  \notag
\end{align}
Finally we get the functional equation of $Z(s)$ from~\eqref{eqn:pre_functional_eqn_III} as follows.

\begin{theorem}[The functional equation of the Selberg zeta function]
  \label{thm:func_eq_Selberg_zeta}
  Let $Z(s)$ be the Selberg zeta function for a discrete subgroup $\Gamma$ without parabolic elements
  in $\PSLR$. Then $Z(s)$ satisfies the following equality:
\begin{align}
  \frac{d}{ds}\log\frac{Z(s)}{Z(1-s)}
  &=\mu(\mathcal{D})(s-1/2)\tan\pi(s-1/2) \label{eqn:functional_eqn}\\
  &\quad + \sum_{\{R\}:\mathrm{elliptic}}
  \frac{-\pi}{m(R)\sin\theta(R)}
  \frac{\cos((2\theta(R)-\pi)(s-1/2))}{\cos \pi(s-1/2)}\notag
\end{align}
where $\{R\}$ runs over the conjugacy classes of elliptic elements $R$ in $\Gamma$.
\end{theorem}

\begin{remark}
  If $\Gamma$ has only hyperbolic elements, then the functional equation~\eqref{eqn:functional_eqn}
  can be written as
  \[ Z(s) = Z(1-s) \exp \left( \mu(\mathcal{D}) \int_0^{s-1/2} v \tan \pi v \,dv \right), \]
  which coincides with the usual functional equation of $Z(s)$.
  For example, we refer to~\cite[Equality~(3.4)]{Selberg56:HarmonicAnalysis}
  and~\cite[Theorem~4.2 in p.~73]{Hejhal76:SelbergTraceFormulaVolI}.
\end{remark}

\section{The higher-dimensional Reidemeister torsion and the dynamical zeta functions}
\label{sec:main_section}
We start with the construction of an $\SL[n]$-representation
of $\pi_1(\Gamma \backslash \PSLR)$
and discuss its properties
in Subsection~\ref{subsec:SLnRepSeifert}.
Subsection~\ref{subsec:Ruelle_SL2N} gives 
the absolute value at zero of the Ruelle zeta function explicitly.
We will show the equality between
the absolute value at zero of the Ruelle zeta function
and the Reidemeister torsion in Subsection~\ref{subsec:Ruelle_Rtorsion_eq}
and discuss the asymptotic behavior of the Reidemeister torsion through the Ruelle zeta function in Subsection~\ref{subsec:relation_asymptotic}.

\subsection{$\SL[n]$-representations for $\pi_1 (\Gamma \backslash \PSLR)$}
\label{subsec:SLnRepSeifert}
For $M = \Gamma \backslash \PSLR$, we will see that there exist some natural representations 
\[\rho_{2N} : \pi_1 M \xrightarrow{\rho} \SLR \xrightarrow{\sigma_{2N}} \SL[2N]\]
such that $\rho_{2N}(h) = - \I_{2N}$ for all $N \geq 1$ as follows.

We denote by $\TPSLR$ the universal cover of $\PSLR$.
The covering projection from $\TPSLR$ onto $\PSLR$ induces
the following exact sequence of groups:
\[
0 \to \Z \to \TPSLR \xrightarrow{p} \PSLR \to 1.
\]
This exact sequence is a central extension of $\PSLR$.
Let $\tilde\Gamma$ be the preimage $p^{-1}(\Gamma)$ of a discrete subgroup $\Gamma$.
We can regard $\tilde\Gamma$ as a subgroup of $\mathrm{Isom}(\TPSLR)$ since $\TPSLR \subset \mathrm{Isom}(\TPSLR)$.
Then $\widetilde{\Gamma}$ acts on $\TPSLR$ freely and 
the quotient $\tilde\Gamma \backslash \TPSLR$ is a Seifert fibered space.
It is known that $\tilde\Gamma \backslash \TPSLR$
turns into the unit tangent bundle $\Gamma \backslash \PSLR$ over $\Sigma = \Gamma \backslash \upperH$
(for the details, see~\cite[\S The geometry of $\TSLR$]{scott83:3manifolds}).
From the identification between $M = \Gamma \backslash \PSLR$ and $\tilde\Gamma \backslash \TPSLR$,
$\tilde\Gamma$ is the fundamental group $\pi_1(M)$ of $M=\Gamma \backslash \PSLR$.

Since the $2$--fold cover $\SLR$ over $\PSLR$ has the same universal cover $\TPSLR$, 
we can take an $\SLR$-representation $\rho$ of $\pi_1(M)$ as
the restriction of the projection from $\TPSLR$ onto $\SLR$ to $\tilde\Gamma = \pi_1(M)$.
The fundamental group $\tilde \Gamma$ is the central extension such that the following
commutative diagram holds:
\[
\xymatrix@R=10pt{
  0 \ar[r] & \Z \ar[r]\ar@{=}[d] & \tilde \Gamma \ar[r] \ar[d]& \Gamma \ar[r]\ar[d] & 1 \\
  0 \ar[r] & \Z \ar[r]           & \TPSLR \ar[r]        & \PSLR \ar[r]  & 1. \\
}
\]
According to~\cite[Proofs of Theorem~1 and 2]{JankinsNeumannHomomorphisms},
we can describe the image of the generators $h$ and $q_j$
in the presentation~\eqref{eqn:presentation_pi1_Seifert}
of $\tilde \Gamma = \pi_1(M)$ under the restriction of the projection from $\TPSLR$ onto $\SLR$
as in the next proposition.

\begin{proposition}
  \label{prop:natural_rep_M}
  Let the Seifert index of $M = \tilde \Gamma \backslash \TPSLR$ be 
  $(g; (1, 2g-2), (\alpha_j, \alpha_j-1), \ldots, (\alpha_m, \alpha_m-1))$
  and $h$ denote the homotopy class of a regular fiber in the Seifert fibration of $M$.
  If $\rho$ is an $\SLR$-representation of $\pi_1(M) = \tilde \Gamma$ given by
  the restriction of the projection from $\TPSLR$ onto $\SLR$ to $\tilde\Gamma$,
  then the images of $h$ and $q_j$ in~\eqref{eqn:presentation_pi1_Seifert}
  and a closed orbit $\gamma$ of the geodesic flow $\phi_t$ by $\rho$ satisfy that
  \begin{itemize}
  \item[]
    $\rho(h) = -\I$,
  \item[]
    $\rho(q_j)$  is conjugate to
    $\begin{pmatrix}
    \cos \frac{(1-\alpha_j) \pi}{\alpha_j} & -\sin \frac{(1-\alpha_j) \pi}{\alpha_j} \\
    \sin \frac{(1-\alpha_j) \pi}{\alpha_j} &  \cos \frac{(1-\alpha_j) \pi}{\alpha_j}
    \end{pmatrix}$,
  \item[]
    $\rho(\gamma)$
    is conjugate to
    $\begin{pmatrix}
    e^{\ell(\gamma)/2} & 0 \\
    0 &  e^{-\ell(\gamma)/2}
    \end{pmatrix}$
  \end{itemize}
  where $\ell(\gamma)$ is the period of the closed orbit $\gamma$.
\end{proposition}

\begin{proof}
  Every regular fiber in the Seifert fibration of $M$
  gives a central element $h$ generating the center ($\simeq \Z$) of $\pi_1(M)$.
  It follows from the construction that
  the center $\langle h \rangle$ of $\pi_1(M)$ coincides with the center of $\TPSLR$.
  We can see that $\rho(h) = -\I$.

  The center $\Z$ of $\TPSLR$ is contained in the component $\R$ in $\TPSLR \simeq \upperH \times \R$
  (homeomorphically, note that $\TPSLR$ is not isometric to $\upperH \times \R$).
  The projection of $\TPSLR$ onto $\SLR$ sends the central element $m \in \Z$ of $\TPSLR$ 
  to $(-\I)^m$ in $\SLR$ and a rational number $m/n$ in $\R$ of $\TPSLR$ to
  an $\SLR$-matrix conjugate to the rotation matrix with angle $m\pi /n$.
  Together with the relation $q_j^{\alpha_j}h^{\alpha_j-1} = 1$, 
  we can also see that $\rho$ sends $q_j$ to an $\SLR$-matrix conjugate to 
  the rotation matrix with angle $(1-\alpha_j)\pi / \alpha_j$.
  
  Finally every closed orbit of the geodesic flow $\phi_t$ is induced by
  the action of
  the one parameter group:
  \[
  \left\{\left. \begin{pmatrix} e^{t/2} & 0 \\ 0 & e^{-t/2} \end{pmatrix} \,\right|\, t \in \R \right\}
  \]
  which gives
  the flow of the vector field generated by $\begin{pmatrix} 1/2 & 0 \\ 0 & -1/2 \end{pmatrix}$
  of $\sllR$ on $\PSLR$.
  There exist the lifts of 
  the vector field to the covering spaces $\TPSLR$ and $\SLR$.
  The image $\rho(\gamma)$ is given by the projection to the flow
  of the vector field generated by the same element of the Lie algebra 
  on $\SLR$.
  Hence the image $\rho(\gamma)$ is conjugate to a hyperbolic element
  with positive trace.
\end{proof}

We also review the irreducible $n$-dimensional representation $\sigma_n$ of $\SL$.
It is known that every irreducible $n$-dimensional representation over $\C$ is
equivalent to the following action of $\SL$ on the space of homogeneous polynomials of degree $n-1$.
Let $V_n$ be the vector space consisting of homogeneous polynomials with coefficient $\C$
of degree $n-1$ in $x$ and $y$, \ie
\[
V_n = \mathrm{span}_{\C} \langle x^{n-1}, x^{n-2}y, \ldots, xy^{n-2}, y^{n-1}\rangle.
\]
The Lie group $\SL$ acts on $V_n$ as
\[
\begin{pmatrix}
  a & b \\
  c & d
\end{pmatrix}
\cdot p(x, y) =
p(\begin{pmatrix}
  a & b \\
  c & d
\end{pmatrix}^{-1}
\begin{pmatrix}x \\ y \end{pmatrix})
= p(dx-by, -cx+ay)
\]
for
$\begin{pmatrix}
  a & b \\
  c & d
\end{pmatrix} \in \SL$.
This action defines an irreducible $\SL[n]$-representation of $\SL$
which is denoted by $\sigma_n$.

\begin{remark}
  \label{remark:sigma_n}
  We give several remarks on $\sigma_n$ needed later.
  \begin{itemize}
  \item
    The vector space $V_n$ is also referred to as the $(n-1)$-th {\it symmetric product} of $\C^2$
    with the standard action of $\SL$.
  \item
    If an $\SL$-matrix $A$ has the eigenvalues $a^{\pm 1}$,
    then $\sigma_n(A)$ has the eigenvalues $a^{n-1}, a^{n-3}, \ldots, a^{-(n-3)}, a^{-(n-1)}$
    (the exponents are decreasing by $2$).
    In the case of $n=2N$,
    the eigenvalues of $\sigma_{2N}(A)$ are given by 
    $\{a^{\pm (2N-1)}, a^{\pm (2N-3)}, \ldots, a^{\pm 1}\}$.
    \end{itemize}
\end{remark}

\begin{definition}[$2N$-dimensional representation of $\pi_1(\Gamma \backslash \PSLR)$]
  We define the $\SL[2N]$-representation $\rho_{2N}$ as the composition of $\sigma_{2N}$
  with the restriction of $\TPSLR \to \SLR$ to $\pi_1(\Gamma \backslash \PSLR)$, \ie
  \[ \rho_{2N} :
  \pi_1(\Gamma \backslash \PSLR) \xrightarrow{\rho} \SLR (\subset \SL)
  \xrightarrow{\sigma_{2N}} \SL[2N].\]
\end{definition}

\begin{proposition}
  \label{prop:rho_2N_h}
  The $\SL[2N]$-representation $\rho_{2N}$ sends $h$ to $-\I_{2N}$.
\end{proposition}
\begin{proof}
  By Proposition~\ref{prop:natural_rep_M} and Remark~\ref{remark:sigma_n},
  $\rho$ sends $h$ to $-\I$ and
  the irreducible representation $\sigma_{2N}$ of $\SL$ sends $-\I$ to $-\I_{2N}$.
  Hence the $\SL[2N]$-representation $\rho_{2N} = \sigma_{2N} \circ \rho$ also sends $h$ to $-\I_{2N}$.
\end{proof}

It follows from Proposition~\ref{prop:rho_2N_h} and \cite[Proposition~4.1]{Yamaguchi:asymptoticsRtorsion}
that our representations $\rho_{2N}$ define the higher-dimensional Reidemeister torsion of $M = \Gamma \backslash \PSLR$ discussed in~\cite{Yamaguchi:asymptoticsRtorsion}.
For the details on the higher-dimensional Reidemeister torsion, see~\cite{Yamaguchi:asymptoticsRtorsion}.
\begin{corollary}
  Each representation $\rho_{2N}$ of $\pi_1(M)$ defines
  an acyclic local system of $M$ with the coefficient $V_{2N}$
  and the higher-dimensional Reidemeister torsion $\Tor{M}{\rho_{2N}}$ of $M$.
\end{corollary}

\subsection{The Ruelle zeta function for $\SL[2N]$-representations}
\label{subsec:Ruelle_SL2N}
We consider the Ruelle zeta function with a non-unitary representation $\rho_{2N}$
of $\pi_1(\Gamma \backslash \PSLR)$ for a discrete subgroup $\Gamma$ in $\PSLR$.
The Ruelle zeta function with the $2N$-dimensional representation $\rho_{2N}$ is also
expressed as the ratio of the Selberg zeta function.
Moreover, the value of our Ruelle zeta function at zero has an integral expression
given by the functional equation of the Selberg zeta function.
We will observe the resulting integral expression in detail.

Here and subsequently,
we often regard the unit tangent bundle $\Gamma \backslash \PSLR$
over a hyperbolic orbifold $\Sigma = \Gamma \backslash \upperH$
as $\tilde \Gamma \backslash \TPSLR$ where $\tilde \Gamma = p^{-1}(\Gamma)$
by the projection $p: \TPSLR \to \PSLR$.

\begin{proposition}
  \label{prop:Ruelle_Selberg_at_s}
  Suppose that $M$ is $\Gamma \backslash \PSLR (= \tilde \Gamma \backslash \TPSLR)$ and
  $\rho$ denotes the $\SLR$-representation of $\pi_1(M) (= \tilde \Gamma)$ given by
  the restriction of the projection from $\TPSLR$ onto $\SLR$ to $\tilde\Gamma$.
  Then the Ruelle zeta function for $M$ with
  the non-unitary representation $\rho_{2N} = \sigma_{2N} \circ \rho$ satisfies 
  \[R_{\rho_{2N}}(s) = \frac{Z(s-N+1/2)}{Z(s+N+1/2)}.\]
  Moreover,
  when we set $\eta(s)=Z(s) / Z(1-s)$, the value of Ruelle zeta function at zero is expressed as
  \begin{align*}
    R_{\rho_{2N}}(0)
    &= \frac{Z(-N+1/2)}{Z(N+1/2)} \\
    &= \eta(N+1/2)^{-1}.
  \end{align*}
\end{proposition}
\begin{proof}
  By Proposition~\ref{prop:natural_rep_M} and Remark~\ref{remark:sigma_n},
  the Ruelle zeta function $R_{\rho_{2N}}(s)$ is expressed as
  \begin{align*}
    R_{\rho_{2N}}(s)
    &= \prod_{\gamma:\textrm{prime}} \det(\I - \rho_{2N}(\gamma) e^{-s\ell(\gamma)}) \\
    &= \prod_{\gamma:\textrm{prime}} \prod_{k=1}^N (1-e^{-(s-(2k-1)/2)\ell(\gamma)})(1-e^{-(s+(2k-1)/2)\ell(\gamma)})\\
    &= \prod_{k=1}^N \prod_{\gamma:\textrm{prime}}  (1-e^{-(s-(2k-1)/2)\ell(\gamma)})(1-e^{-(s+(2k-1)/2)\ell(\gamma)})\\
    &= \prod_{k=1}^N R\left(s-\frac{2k-1}{2}\right) R\left(s+\frac{2k-1}{2}\right) \\
    &= \frac{Z(s-1/2)}{Z(s+1/2)} \cdots \frac{Z(s-(2N-1)/2)}{Z(s-(2N-3)/2)}
       \frac{Z(s+1/2)}{Z(s+3/2)} \cdots \frac{Z(s+(2N-1)/2)}{Z(s+(2N+1)/2)}\\
    &= \frac{Z(s-N+1/2)}{Z(s+N+1/2)}.
  \end{align*}
  The second equality from the last follows from the classical relation~\eqref{eqn:classical_Ruelle_Selberg}
  between the Ruelle and Selberg zeta functions.
\end{proof}

We will examine the absolute value at zero of the Ruelle zeta function $R_{\rho_{2N}}(s)$
more closely.
By calculating the integral expression of $\eta(N + 1/2) = R_{\rho_{2N}}(0)^{-1}$,
we have the explicit formula of $|R_{\rho_{2N}}(0)|^{-1}$.
\begin{theorem}
  \label{thm:Ruelle_zero}
  Under the same assumptions as Proposition~\ref{prop:Ruelle_Selberg_at_s}, 
  the absolute value at zero of the Ruelle zeta function is expressed as follows.
  \begin{align}
    |R_{\rho_{2N}}(0)|^{-1}
    &= |\eta(N+1/2)| \notag \\
    &= \exp \left[ \mu(\mathcal{D})  \left( -\frac{N}{\pi} \log 2 \right) \right] \notag \\
    &\qquad \cdot
    \exp \left[
      \sum_{j=1}^m \left(
      \log \prod_{k=1}^N (1 - e^{(2k-1)\theta(q_j)\iu}) (1 - e^{-(2k-1)\theta(q_j)\iu})
      - \frac{2N}{m(q_j)}\log2
      \right)
      \right]
    \label{eqn:Ruelle_zero_elliptic}
  \end{align}
\end{theorem}

To prove Theorem~\ref{thm:Ruelle_zero} we start with the integral expression of $\eta(s)$.
By $\eta(1/2)=1$ and the functional equation~\eqref{eqn:functional_eqn}
of the Selberg zeta function, we can express $\eta(s)$ as follows.
\begin{align*}
  \eta(s)
  &=\exp\left( \int_{1/2}^s \frac{d}{dz} \log \eta(z)\,dz \right)\\
  &=\exp\left(\mu(\mathcal{D})\int_{1/2}^s (z-1/2)\tan \pi(z-1/2)\,dz \right.\\
  &\quad \left. +
  \sum_{\{R\}:\mathrm{elliptic}}
  \int_{1/2}^s \frac{-\pi}{m(R)\sin\theta(R)}
  \frac{\cos((2\theta(R)-\pi)(z-1/2))}{\cos \pi(z-1/2)}\, dz \right), \\
  \intertext{put $\xi=z-1/2$,}
  &=\exp\left(\mu(\mathcal{D})\int_{0}^{s-1/2} \xi \tan \pi \xi\,d\xi \right.\\
  &\quad \left. +
  \sum_{\{R\}:\mathrm{elliptic}}
  \int_{0}^{s-1/2} \frac{-\pi}{m(R)\sin\theta(R)}
  \frac{\cos((2\theta(R)-\pi)\xi)}{\cos \pi\xi}\, d\xi \right). \\
\end{align*}
Summarizing, we have the following integral expression of $\eta(N+1/2) = R_{\rho_{2N}}^{-1}(0)$.
\begin{proposition}
  \label{prop:eta_integral}
  \begin{align*}
    R^{-1}_{\rho_{2N}}(0) &= \eta(N + 1/2)\\
    &=
    \exp\left(
    \mu(\mathcal{D})\int_{0}^{N} \xi \tan \pi \xi\,d\xi
    + \sum_{\{R\}:\mathrm{elliptic}}
    \int_{0}^{N} \frac{-\pi}{m(R)\sin\theta(R)}
    \frac{\cos((2\theta(R)-\pi)\xi)}{\cos \pi\xi}\, d\xi \right).
  \end{align*}
\end{proposition}
Then the absolute value $|\eta(N+1/2)|$ is determined by the real part of
the exponent:
\[
\re \left( \mu(\mathcal{D})\int_{0}^{N} \xi \tan \pi \xi\,d\xi 
+ \sum_{\{R\}:\mathrm{elliptic}} \int_{0}^{N} \frac{-\pi}{m(R)\sin\theta(R)}
  \frac{\cos((2\theta(R)-\pi)\xi)}{\cos \pi\xi}\, d\xi \right)
\]
We calculate these integrals
in Propositions~\ref{prop:contribution_identity} and~\ref{prop:contribution_elliptic}
which arise from the identity element and elliptic elements in $\Gamma$.
Theorem~\ref{thm:Ruelle_zero} follows from Propositions~\ref{prop:eta_integral},
\ref{prop:contribution_identity} and~\ref{prop:contribution_elliptic}
together with the presentation~\eqref{eqn:presentation_Gamma} of $\Gamma$.

\begin{proposition}
  \label{prop:contribution_identity}
  The contribution of the identity element to the exponent of $|\eta(N+1/2)|$
  is expressed as
  \[
  \re \left( \mu(\mathcal{D})\int_{0}^{N} \xi \tan \pi \xi\,d\xi \right)
  = \mu(\mathcal{D}) \left( - \frac{N}{\pi} \log 2\right).
  \]
\end{proposition}
\begin{proof}
  The real part of the path integral turns out to be 
  \begin{align*} 
    \re \left( \int_{0}^{N} \xi \tan \pi \xi\,d\xi \right)
    &= \re \left(\int_{0}^{N} \xi \tan \pi \xi \,d\xi\right)\\
    &= \left[\xi \frac{-1}{\pi} \log |\cos \pi\xi| \right]_0^N
    + \frac{1}{\pi}~\mathrm{p.v.} \int_0^N \log |\cos \pi\xi| \xi \\
    &= \frac{N}{\pi}~\mathrm{p.v.} \int_0^1 \log |\cos \pi\xi| \xi \\
    &= - \frac{N}{\pi} \log 2.
  \end{align*}
\end{proof}

Next we proceed to the contribution of elliptic elements in $\Gamma$
to the exponent of $|\eta(N+1/2)|$.
The set of conjugacy classes of elliptic elements is divided into
the subsets
\[
\{ \{R_0\}, \{R_0^2\}, \ldots, \{R_0^{m_0-1}\}\}
\]
of conjugacy classes which have the representatives of the same order $m_0$.
\begin{proposition}
  \label{prop:contribution_elliptic}
  Suppose that an elliptic element $R_0$ satisfies $\trace R_0 = 2\cos\theta(R_0)$
  with $\theta(R_0) = \pi / m_0 \,(m_0 \geq 2)$.
  Then the sum of
  \[
  \re \left( \int_{0}^{N} \frac{-\pi}{m(R)\sin\theta(R)}
  \frac{\cos((2\theta(R)-\pi)\xi)}{\cos \pi\xi}\, d\xi \right)
  \]
  over $\{ \{R_0\}, \{R_0^2\}, \ldots, \{R_0^{m_0-1}\}\}$
  is given by
  \[
  \log \prod_{k=1}^N (1 - e^{(2k-1)\pi \iu / m_0})(1 - e^{-(2k-1)\pi \iu / m_0})
  -\frac{2N}{m_0} \log 2.
  \]
\end{proposition}
  
\begin{lemma}
  \label{lemma:int_1st}
  The contribution of the elliptic elements $R_0^p (1\leq p \leq m_0 -1)$
  to the exponent of $|\eta(N+1/2)|$
  over the 1st interval $[0, 1]$ is expressed as 
  \[
  \re \left(
  \int_0^1
  \sum_{p=1}^{m_0-1}
  \frac{
    -\pi
  }{
    m_0 \sin\frac{p\pi}{m_0}
  }
  \frac{
    \cos(2\frac{p\pi}{m_0}-\pi)\xi
  }{
    \cos\pi\xi
  }
  \,d\xi
  \right)
  =
  2 \log |1 - e^{-\pi \iu / m_0}| - \frac{2}{m_0} \log2.
  \]
\end{lemma}
\begin{proof}
  Put $z=e^{\pi \xi \iu / m_0}$ and $z_0 = e^{\pi \iu / m_0}$.
  We can rewrite the integral as 
  \begin{align}
    \re \left(
    \int_0^1
    \sum_{p=1}^{m_0-1}
  \frac{
    -\pi
  }{
    m_0 \sin\frac{p\pi}{m_0}
  }
  \frac{
    \cos(2\frac{p\pi}{m_0}-\pi)\xi
  }{
    \cos\pi\xi
  }
  \,d\xi
  \right)
  &=
  \re \left(
  \int_1^{z_0}
  \frac{-1}{\iu z} 
  \sum_{p=1}^{m_0-1}
  \frac{
    z^{2p-m_0} + z^{-2p+m_0}
  }{
    \sin \frac{p\pi}{m_0}
    (z^{m_0} + z^{-m_0})
  }
  \,dz
  \right) \notag \\
  &=
  \re
  \int_1^{z_0} \frac{P(z)}{Q(z)} \, dz
  \label{eqn:real_int_P_Q}
  \end{align}
  where $P(z)$ and $Q(z)$ are the polynomials in $z$ defined as 
  \[
  P(z)
  = \frac{-1}{\iu z}
  \sum_{p=1}^{m_0-1}
  \frac{
    z^{2p-m_0} + z^{-2p+m_0}
  }{
    \sin \frac{p\pi}{m_0}
  }
  z^{m_0}, \quad
  Q(z) = (z^{m_0} + z^{-m_0})z^{m_0} = z^{2m_0} + 1.
  \]
  Since the zeros of $Q(z)$ are $\{\pm e^{(2q-1)\pi \iu/(2m_0)} \,|\, q = 1, \ldots, m_0\}$,
  the integrand is decomposed into the sum of the partial fractions. Then we have 
  \begin{align}
    \int_1^{z_0} \frac{P(z)}{Q(z)} \, dz
    &=
    \int_1^{z_0}
    \sum_{z_1: Q(z_1)=0}
    \frac{1}{z-z_1}\frac{P(z_1)}{Q'(z_1)}
    \,dz \notag \\
    &=
    \sum_{z_1: Q(z_1)=0}
    \frac{P(z_1)}{Q'(z_1)}
    \log\frac{z_0-z_1}{1-z_1} \notag \\
    \intertext{by $z_1^{2m_0} = -1$ and $P(z_1)/Q'(z_1) = P(-z_1) / Q'(-z_1)$ }
    &=
    \sum_{z_1}
    \frac{1}{m_0}
    \sum_{p=1}^{m_0-1}\frac{z_1^{2p}-z_1^{-2p}}{z_0^p - z_0^{-p}}
    \left(
    \log\frac{z_0-z_1}{1-z_1}
    + \log\frac{z_0+z_1}{1+z_1}
    \right) \notag \\
    &=
    \sum_{z_1}
    \frac{1}{m_0}
    \sum_{p=1}^{m_0-1}\frac{z_1^{2p}-z_1^{-2p}}{z_0^p - z_0^{-p}}
    \left(
    \log (z_0^2-z_1^2)
    - \log (1-z_1^2)
    \right)
    \label{eqn:integral_I}
  \end{align}
  where $z_1$ in the last sum runs over $\{e^{(2q-1) \pi \iu / (2m_0)} \,|\, q = 1, \ldots, m_0\}$.
  The real part of $\log(z_0^2 - z_1^2)$ in~\eqref{eqn:integral_I}
  is the same as $\log|1-z_1^2 / z_0^2|$ since $z_0$ is a root of unity.
  By putting $w=z_1^2$ in~\eqref{eqn:integral_I}, we can express~\eqref{eqn:real_int_P_Q} as   
  \begin{align*}
    \re
    \int_1^{z_0} \frac{P(z)}{Q(z)} \, dz
    &=
    \sum_{w^{m_0} = -1}
    \frac{1}{m_0}
    \sum_{p=1}^{m_0-1}\frac{w^p - w^{-p}}{z_0^p - z_0^{-p}}
    \left(
    \log \left|1-\frac{w}{z_0^2}\right|
    - \log|1-w|
    \right) \\
    &=
    \sum_{w^{m_0} = -1}
    \frac{1}{m_0}
    \left(
    \log \left|1-\frac{w}{z_0^2}\right|
    \sum_{p=1}^{m_0-1}\frac{w^p - w^{-p}}{z_0^p - z_0^{-p}}
    \right) \\
    &\qquad- \sum_{w^{m_0} = -1}
    \frac{1}{m_0}
    \left(
    \log|1-w|
    \sum_{p=1}^{m_0-1}\frac{w^p - w^{-p}}{z_0^p - z_0^{-p}}
    \right), \\
    \intertext{since $w/z_0^2$ also satisfies $(w/z_0^2)^{m_0} =-1$,}
    &=
    \frac{1}{m_0}
    \sum_{w^{m_0} = -1}
    \left(
    \log|1-w|
    \sum_{p=1}^{m_0-1}\frac{(z_0^2 w)^p - (z_0^2 w)^{-p} - w^p + w^{-p}}{z_0^p - z_0^{-p}}
    \right)\\
    &=
    \frac{1}{m_0}
    \sum_{w^{m_0} = -1}
    \left(
    \log|1-w|
    \sum_{p=1}^{m_0-1}(z_0^p w^p + z_0^{-p}w^{-p})
    \right)\\
    &=
    \frac{2}{m_0}
    \sum_{w^{m_0} = -1}
    \left(
    \log|1-w|
    \sum_{p=1}^{m_0-1}(z_0 w)^p
    \right),\\
    \intertext{as the sum of $(z_0 w)^p$ turns into $m_0-1$ or $-1$ if $z_0 w = 1$ or not,}
    &=
    \frac{2}{m_0}
    \left(
    (m_0-1)\log|1-z_0^{-1}|
    -
    \sum_{w^{m_0} = -1, w \not = z_0^{-1}}
    \log|1-w|
    \right)\\
    &=
    \frac{2}{m_0}
    \left(
    m_0\log|1-z_0^{-1}|
    -
    \sum_{w^{m_0} = -1}
    \log|1-w|
    \right) \\
    &=
    2\log |1-z_0^{-1}| - \frac{2}{m_0} \log 2
  \end{align*}
  by $\prod_{w^{m_0}=-1} (x-w) = x^{m_0} + 1$.
  This completes the proof.
\end{proof}

\begin{lemma}
  \label{lemma:int_jth}
  The contribution of the elliptic elements $R_0^p (1\leq p \leq m_0 -1)$
  to the exponent of $|\eta(N+1/2)|$
  over the $k$-th interval $[k-1, k]$ is expressed as 
  \begin{equation}
    \label{eqn:integral_jth}
  \re \left(
  \int_{k-1}^k
  \sum_{p=1}^{m_0-1}
  \frac{
    -\pi
  }{
    m_0 \sin\frac{p \pi}{m_0}
  }
  \frac{
    \cos(2\frac{p \pi}{m_0}-\pi)\xi
  }{
    \cos\pi\xi
  }
  \,d\xi
  \right)
  =
  2 \log |1 - e^{-(2k-1)\pi \iu / m_0}| - \frac{2}{m_0} \log2
  \end{equation}
\end{lemma}
\begin{proof}
  This follows in much the same way as in the proof of Lemma~\ref{lemma:int_1st}.
  We show the differences from the proof of Lemma~\ref{lemma:int_1st} under the same notations.
  The integral of the left hand side in~\eqref{eqn:integral_jth} turns into 
  \begin{align*}
    &\re \int_{z_0^{k-1}}^{z_0^k} \frac{P(z)}{Q(z)} \, dz\\
    &=
    \re \sum_{z_1: Q(z_1)=0}
    \frac{P(z_1)}{Q'(z_1)}
    \log \frac{z_0^k - z_1}{z_0^{k-1} - z_1}, \\
    \intertext{by a similar argument in the proof of Lemma~\ref{lemma:int_1st}
      and putting $w = z_1^2 / z_0^{2(k-1)}$,}
    &=
    \sum_{w^{m_0} = -1}
    \frac{1}{m_0}
    \sum_{p=1}^{m_0-1}\frac{(z^{2(k-1)}w)^p - (z^{2(k-1)}w)^{-p}}{z_0^p - z_0^{-p}}
    \left(
    \log \left|1-\frac{w}{z_0^2}\right|
    - \log|1-w|
    \right)\\
    &=\sum_{w^{m_0} = -1}
    \frac{1}{m_0}
    \log \left|1-\frac{w}{z_0^2}\right|
    \sum_{p=1}^{m_0-1}\frac{(z^{2(k-1)}w)^p - (z^{2(k-1)}w)^{-p}}{z_0^p - z_0^{-p}}\\
    &\qquad
    -\sum_{w^{m_0} = -1}
    \frac{1}{m_0}
    \log|1-w|
    \sum_{p=1}^{m_0-1}\frac{(z^{2(k-1)}w)^p - (z^{2(k-1)}w)^{-p}}{z_0^p - z_0^{-p}}\\
    &=
    \frac{1}{m_0}
    \sum_{w^{m_0} = -1}
    \left(
    \log|1-w|
    \sum_{p=1}^{m_0-1}
    \frac{
      (z_0^{2k} w)^p - (z_0^{2k} w)^{-p} - (z_0^{2(k-1)} w)^p + (z_0^{2(k-1)}w)^{-p}
    }{
      z_0^p - z_0^{-p}
    }
    \right)
    \\
    &=
    \frac{1}{m_0}
    \sum_{w^{m_0} = -1}
    \left(
    \log|1-w|
    \sum_{p=1}^{m_0-1}((z_0^{2k-1} w)^p + (z_0^{2k-1} w)^{-p})
    \right) \\
    &= 2\log |1 - z_0^{-(2k-1)}| - \frac{2}{m_0}\log 2
  \end{align*}
  which proves our claim.
\end{proof}

Finally we proceed to show Proposition~\ref{prop:contribution_elliptic}.
\begin{proof}[Proof of Proposition~\ref{prop:contribution_elliptic}]
  By Lemmas~\ref{lemma:int_1st} and \ref{lemma:int_jth}, 
  our sum over the conjugacy classes $\{R_0\}, \{R_0^2\}, \ldots, \{R_0^{m_0-1}\}$
  turns out to be 
  \begin{align*}
    &\sum_{p=1}^{m_0-1} \re  \left( \int_{0}^{N} \frac{-\pi}{m(R_0^p) \sin\theta(R_0^p)}
    \frac{\cos((2\theta(R_0^p)-\pi)\xi)}{\cos \pi\xi}\, d\xi \right)\\
    &=
    2 \sum_{k=1}^N \log|1 - e^{-(2k-1)\pi \iu / m_0}| - \frac{2N}{m_0} \log 2\\
    &=
    \log \prod_{k=1}^N |1 - e^{-(2k-1)\pi \iu / m_0}|^2 - \frac{2N}{m_0} \log 2\\
    &=
    \log \prod_{k=1}^N (1 - e^{(2k-1)\pi \iu / m_0})(1 - e^{-(2k-1)\pi \iu / m_0}) - \frac{2N}{m_0} \log 2.
  \end{align*}  
\end{proof}

\subsection{Relation to the higher-dimensional Reidemeister torsion}
\label{subsec:Ruelle_Rtorsion_eq}
We regard the fundamental group $\pi_1(\Gamma \backslash \PSLR)$
as the subgroup $\tilde \Gamma$ in $\TPSLR$. 
Assume that
the representation $\rho$ of $\pi_1(\Gamma \backslash \PSLR)$ is
the restriction of $\TPSLR \to \SLR$ to $\tilde \Gamma$.
This means that our representation is induced by the geometric description
$\tilde \Gamma \backslash \TSLR$ of $\Gamma \backslash \PSLR$.
We will show that the absolute value at zero of
the Ruelle zeta function $R_{\rho_{2N}}(s)$
for the induced $\SL[2N]$-representation $\rho_{2N}$ coincides with 
the higher-dimensional Reidemeister torsion $\Tor{M}{\rho_{2N}}$
for $M = \Gamma \backslash \PSLR$ with $\rho_{2N}$.

For this purpose, we find out the order of the $\SLR$-matrix $\rho(\ell_j)$
corresponding to the exceptional fiber $\ell_j$ of $M$ in $\SL$.
\begin{lemma}
  \label{lemma:rho_ell}
  The $\SLR$-matrix $\rho(\ell_j)$ for each exceptional fiber $\ell_j$
  is conjugate to
  \[\begin{pmatrix}
    e^{\pi \iu / \alpha_j} & 0 \\
    0 & e^{-\pi \iu /\alpha_j}
  \end{pmatrix}\]
  in $\SL$.
  In particular, the order of $\rho(\ell_j)$ is equal to $2\alpha_j$ in $\SL$.
\end{lemma}
\begin{proof}
  By the relation~\eqref{eqn:exceptional_fiber}, $\rho(\ell_j)$ is expressed as
  $\rho(\ell_j) = \rho(q_j)^{-1} \rho(h)^{-1}$.
  It follows from Proposition~\ref{prop:natural_rep_M} that
  $\rho(q_j)^{-1} \rho(h)^{-1} = -\rho(q_j)^{-1}$ is conjugate to
    $\begin{pmatrix}
    e^{\pi \iu / \alpha_j} & 0 \\
    0 & e^{-\pi \iu /\alpha_j}
  \end{pmatrix}$ in $\SL$.
\end{proof}

It follows from Lemma~\ref{lemma:rho_ell} and Remark~\ref{remark:sigma_n}
that the $\SL[2N]$-matrix $\rho_{2N}(\ell_j)$ has the eigenvalues
$e^{\pm (2N-1) \pi \iu / \alpha_j}, e^{\pm (2N-3) \pi \iu / \alpha_j}, \ldots, e^{\pm\pi \iu / \alpha_j}$.
By~\cite[Proposition~4.8]{Yamaguchi:asymptoticsRtorsion}, 
the Reidemeister torsion for $\Gamma \backslash \PSLR$ and $\rho_{2N}$ is given as follows.
\begin{proposition}
  \label{prop:Rtorsion_explicit}
  The higher-dimensional Reidemeister torsion for $M = \Gamma \backslash \PSLR$ with $\rho_{2N}$
  is expressed as
  \begin{equation}
  \Tor{M}{\rho_{2N}}
  = 2^{-2N(2-2g-m)} \prod_{j=1}^m \prod_{k=1}^N \left(2 \sin \frac{\pi (2k-1)}{2\alpha_j} \right	)^{-2}
  \label{eqn:RtorsionExplicit_value}\\
  \end{equation}
\end{proposition}

We are now in position to show the relation
between the Ruelle zeta function and the Reidemeister torsion.
\begin{theorem}
  \label{thm:Ruelle_Rtorsion_eqn}
  $|R_{\rho_{2N}}(0)| = \Tor{M}{\rho_{2N}}$
\end{theorem}
\begin{proof}
  We have seen $|R_{\rho_{2N}}(0)|$ in Theorem~\ref{thm:Ruelle_zero}.
  The angle $\theta(q_j)$ and the order $m(q_j)$ in the equality~\eqref{eqn:Ruelle_zero_elliptic}
  satisfy $\theta(q_j) = \pi / \alpha_j$ and $m(q_j) = \alpha_j$
  (note that $m(q_j)$ is the order of $\rho(q_j)$ in $\PSLR$).
  By the Gauss--Bonnet theorem $\mu(\mathcal{D})  = -2\pi \chi^{\mathrm{orb}}(\Sigma)$,
  we can rewrite the equality~\eqref{eqn:Ruelle_zero_elliptic} as
  \begin{align*}
    |R_{\rho_{2N}}(0)|^{-1}
    &= \exp \left[ \mu(\mathcal{D})  \left( -\frac{N}{\pi} \log 2 \right) \right] \\
    &\qquad \cdot
    \exp \left[ 
      \sum_{j=1}^m \left(
      \log \prod_{k=1}^N
      (1 - e^{(2k-1) \pi \iu / \alpha_j}) (1 - e^{-(2k-1) \pi \iu /  \alpha_j })
      - \frac{2N}{\alpha_j}\log2
      \right)
      \right] \\
    &=
    \exp \left[ 2N \chi^{\mathrm{orb}}(\Sigma) \log 2 \right] \\
    &\qquad \cdot
    \exp \left[ 
      \sum_{j=1}^m \left(
      \log \prod_{k=1}^N
      \left( 2 \sin \frac{\pi(2k -1)}{2\alpha_j}\right)^2
      - \frac{2N}{\alpha_j}\log2
      \right)
      \right], \\
    \intertext{replacing $\chi^{\mathrm{orb}}(\Sigma)$ with $2-2g - \sum_{j=1}^m (\alpha_j - 1) / \alpha_j$,}
    &= \exp \left[  2N \left(2-2g - \sum_{j=1}^m \frac{\alpha_j - 1}{\alpha_j} \right) \log 2 \right]\\
    &\qquad \cdot
    \exp \left[ 
      \sum_{j=1}^m \left(
      \log \prod_{k=1}^N
      \left( 2 \sin \frac{\pi(2k -1)}{2\alpha_j}\right)^2
      - \frac{2N}{\alpha_j}\log2
      \right)
      \right] \\
    &= \exp \left[  2N (2-2g - m)\log 2 \right]
    \exp \left[ 
      \sum_{j=1}^m \left(
      \log \prod_{k=1}^N
      \left( 2 \sin \frac{\pi(2k -1)}{2\alpha_j}\right)^2
      \right)
      \right],\\
    \intertext{comparing with \eqref{eqn:RtorsionExplicit_value},}
    &= \Tor{M}{\rho_{2N}}^{-1}.
  \end{align*}
\end{proof}

The Reidemeister torsion $\Tor{M}{\rho_{2N}}$ in Proposition~\ref{prop:Rtorsion_explicit}
is derived from the following equality:
\begin{equation}
  \Tor{M}{\rho_{2N}}
  = \det(\I - \rho_{2N}(h))^{-(2-2g-m)} \prod_{j=1}^m \det (\I - \rho_{2N}(\ell_j))^{-1}.
  \label{eqn:RtorsionExplicit_fibers}
\end{equation}
Note that $\det(\I - \rho_{2N}(h)) = 2^{2N}$.
The contribution of the identity element to $|R_{\rho_{2N}}(0)|^{-1}$ is also expressed as
\begin{align}
  \exp \left[ \mu(\mathcal{D})  \left( -\frac{N}{\pi} \log 2 \right) \right]
  &= \exp \left[ 2N \chi^{\mathrm{orb}}(\Sigma) \log 2  \right] \label{eqn:contribution_identity}\\
  &= \det(\I - \rho_{2N}(h))^{2- 2g - \sum_{j=1}^m (\alpha_j -1)/\alpha_j}. \notag
\end{align}
The contribution of elliptic elements to $|R_{\rho_{2N}}(0)|^{-1}$ is expressed as
\begin{align}
 & \exp \left[ 
    \sum_{j=1}^m \left(
    \log \prod_{k=1}^N
    (e^{(2k-1)\pi \iu / \alpha_j} - 1) (e^{-(2k-1) \pi \iu /  \alpha_j } - 1)
    - \frac{2N}{\alpha_j}\log2
    \right)
    \right] \label{eqn:contribution_elliptic} \\
  &= \prod_{j=1}^m \det (\I -\rho_{2N}(\ell_j)) \det (\I - \rho_{2N}(h)) ^{-1/\alpha_j}.
  \notag
\end{align}

One can find the common factors $\prod_{j=1}^m \det (\I - \rho_{2N}(h)) ^{-1/\alpha_j}$
in the both of the contributions. The common factors cancel out in the product and
the remain factors give the Reidemeister torsion as in~\eqref{eqn:RtorsionExplicit_fibers}.
We will see the contribution of the identity element only affects the leading coefficient
in the asymptotic behavior of $\log |\Tor{M}{\rho_{2N}}|$ as $N \to \infty$ in the next Subsection.

\begin{remark}
Our manifold $M = \Gamma \backslash \PSLR$ is regarded as
$\tilde \Gamma \backslash \TPSLR$ and $\tilde \Gamma = p^{-1}(\Gamma) \subset \TPSLR$ as 
a subgroup of $\mathrm{Isom}\,\TPSLR$.
Hence $M$ is geometric and
the representation $\rho_{2N} = \sigma_{2N} \circ \rho$ is obtained
from the geometric description of $M$.
However not all discrete subgroups $G \subset \mathrm{Isom}\,\TPSLR$
with Seifert fibered spaces $G \backslash \TPSLR \to \Gamma \backslash \upperH$ 
satisfy $G = p^{-1} (\Gamma)$ under the projection $p : \TPSLR \to \PSLR$.
See~\cite[Theorem~4.15]{scott83:3manifolds} for the details.
\end{remark}

\subsection{Asymptotic behavior of the Reidemeister torsion via Ruelle zeta function}
\label{subsec:relation_asymptotic}
In the previous section, we have seen that the Reidemeister torsion is expressed as
the absolute value at zero of the Ruelle zeta function.
We can also see the asymptotic behavior of the Reidemeister torsion for
$2N$-dimensional representations
from the viewpoint of the Ruelle zeta function
as $N$ goes to infinity.

First we see that the contribution~\eqref{eqn:contribution_elliptic} of elliptic elements
to the Ruelle zeta function
vanishes in the asymptotic behavior of $\log|R_{\rho_{2N}}(0)|$ when $N$ goes to infinity.
\begin{lemma}
  \label{lemma:elliptic_term_vanishes}
  The contribution of elliptic elements to $(\log|R_{\rho_{2N}}(0)|) / (2N)$ converges to zero
  as $N \to \infty$, that is,
  \[ \frac{1}{2N}
  \left(-\log \hbox{\rm of}~\eqref{eqn:contribution_elliptic}\right)
  \xrightarrow{N \to \infty} 0. \]
\end{lemma}
\begin{proof}
  The contribution of elliptic elements to $\log |R_{\rho_{2N}}(0)|$
  is also given by 
  \begin{align*}
    &-\log \left( \prod_{j=1}^m \det (\I -\rho_{2N}(\ell_j)) \det (\I - \rho_{2N}(h)) ^{-1/\alpha_j} \right)\\
    &= -\sum_{j=1}^m \log \det (\I -\rho_{2N}(\ell_j))
    + \sum_{j=1}^m \frac{1}{\alpha_j} \log \det (\I - \rho_{2N}(h)) \\
    &= -\sum_{j=1}^m \log \det (\I -\rho_{2N}(\ell_j))
    + \sum_{j=1}^m \frac{2N}{\alpha_j} \log 2
  \end{align*}
  
  By Lemma~\ref{lemma:rho_ell},
  our $\SL$-representation $\rho$ sends each exceptional fiber $\ell_j$ to
  an $\SL$-matrix of order $2\alpha_j$.
  Since $\det (\I -\rho_{2N}(\ell_j) )^{-1}$ is regarded as
  the Reidemeister torsion for the exceptional fiber $\ell_j$ and
  the restriction of $\rho_{2N}$, 
  it follows from \cite[Proposition~3.8]{Yamaguchi:asymptoticsRtorsion}
  that
  \[
  \lim_{N \to \infty} \frac{\log \det (\I -\rho_{2N}(\ell_j))^{-1} }{2N}
  =\frac{-\log 2}{\alpha_j}.
  \]
  The contribution of elliptic elements to the limit of $(\log |R_{\rho_{2N}}(0)|)/(2N)$ is given by 
  \begin{align*}
    &\lim_{N \to \infty} \frac{1}{2N}
    \left(
    - \sum_{j=1}^m \log \det (\I -\rho_{2N}(\ell_j))
    + \sum_{j=1}^m \frac{2N}{\alpha_j} \log 2
    \right)\\
    &= 
    \sum_{j=1}^m \lim_{N \to \infty} \frac{\log \det (\I -\rho_{2N}(\ell_j))^{-1} }{2N}
    + \sum_{j=1}^m \frac{1}{\alpha_j} \log 2 \\
    &= 0.
  \end{align*}
\end{proof}

We can recover the limit of the leading coefficient in $\log |\Tor{M}{\rho_{2N}}|$ 
from the asymptotic behavior of $|R_{\rho_{2N}}(0)|$.
\begin{proposition}
  \label{prop:limit_Ruelle}
  The leading coefficient of $(\log |\Tor{M}{\rho_{2N}}|) / (2N)$
  is determined by the contribution~\eqref{eqn:contribution_identity}
  of the identity element to $R_{\rho_{2N}}(0)$ as $N \to \infty$.
  That is,
  \begin{equation}
    \label{eqn:limit_lead_coeff}
    \lim_{N \to \infty}\frac{\log |\Tor{M}{\rho_{2N}}|}{2N}
    = \frac{\mathrm{Area}(\Sigma)}{2\pi} \log 2.
  \end{equation}
  Therefore we can derive the area of the base orbifold $\Sigma$
  from the limit of the leading coefficient
  in $\log |\Tor{M}{\rho_{2N}}|$.
\end{proposition}
\begin{proof}
  It follows from
  Theorem~\ref{thm:Ruelle_Rtorsion_eqn} and
  Lemma~\ref{lemma:elliptic_term_vanishes} that 
  \begin{align*}
    \lim_{N \to \infty} \frac{\log |\Tor{M}{\rho_{2N}}|}{2N} 
    &= \lim_{N \to \infty} \frac{\log|R_{\rho_{2N}}(0)|}{2N} \\
    &= \lim_{N \to \infty} \frac{1}{2N}
    \log \exp \left[
      -\mu(\mathcal{D})\left(-\frac{N}{\pi} \log 2\right)
      \right] \\
    &= \frac{\mu(\mathcal{D})}{2\pi} \log 2.
  \end{align*}  
\end{proof}
Together with the Gauss--Bonnet theorem
$\mathrm{Area}(\Sigma) = \mu(\mathcal{D}) = -2\pi \chi^{\mathrm{orb}}$
for the hyperbolic orbifold $\Sigma$,
the equality~\eqref{eqn:limit_lead_coeff} turns out to be
\[
\lim_{N \to \infty} \frac{\log |\Tor{M}{\rho_{2N}}|}{2N}
=- \chi^{\mathrm{orb}} \log 2.
\]
It was showed in~\cite{Yamaguchi:asymptoticsRtorsion}
that there exist finitely many possibilities in the limits of the leading coefficient
in $\log|\Tor{M}{\rho_{2N}}|$ for sequences
of $\SL[2N]$-representations starting with $\SL$-representations sending $h$ to $-\I$ and
the maximum in those possibilities is given by $- \chi^{\mathrm{orb}} \log 2$.

Proposition~\ref{prop:limit_Ruelle} shows that
the limit of the leading coefficient in $\log|\Tor{M}{\rho_{2N}}|$
for the sequence of $\SL[2N]$-representations $\rho_{2N}$ induced from the geometric description
$M = \tilde \Gamma \backslash \TPSLR$ realizes
the upper bound of limits of the leading coefficients
in the logarithm of the Reidemeister torsion of $M$.

\begin{corollary}
  \label{cor:limit_leadcoeff}
  The limit of $(\log|\Tor{M}{\rho_{2N}}|)/(2N)$ attains
  the maximum $- \chi^{\mathrm{orb}} \log 2$.
\end{corollary}

\section*{Acknowledgments}
The author was supported by JSPS KAKENHI Grant Number $17K05240$.
\bibliographystyle{amsalpha}
\bibliography{GeodesicFlowZetaFunctionRtorsion}

\end{document}